\documentclass[12pt, reqno, a4paper]{amsart}
\usepackage{amsmath,mathrsfs}
\usepackage{amssymb}
\usepackage{color}
\usepackage{calligra}
\usepackage{dsfont}
\usepackage{pstricks,pst-node}
\usepackage[latin1]{inputenc}

\newcommand\NoBlackBoxes{\global\overfullrule0pt}
\NoBlackBoxes

\newcommand{\N}{\mathbb{N}}
\newcommand{\OmegaSW}{{\Omega^\text{sw}}}
\newcommand{\OmegaP}{{\Omega_+}}
\newcommand{\OmegaM}{\Omega_-}

\makeatletter
\let\serieslogo@\relax
\let\@setcopyright\relax

\newtheorem{definition}{Definition}[section]
\newtheorem{theorem}[definition]{Theorem}
\newtheorem{lemma}[definition]{Lemma}

\newtheorem{corollary}[definition]{Corollary}

\newenvironment{remark}[1][Remark]{\begin{trivlist}
\item[\hskip \labelsep {\bfseries #1}]}{\end{trivlist}}

\renewcommand{\P}{{\mathbb{P}}}
\newcommand{\E}{{\mathbb{E}}}
\newcommand{\R}{{\mathbb{R}}}

\newcommand{\Gap}{\operatorname{Gap}}
\newcommand{\follows}{\Longrightarrow}
\newcommand{\Var}{\mathbb{V}}

\newcommand{\goesto}{\longrightarrow}
\newcommand{\Klow}{K_{\text{low}}}
\newcommand{\A}{\mathcal{A}}
\newcommand{\Pst}{P_{\text{st}}}

\newcommand{\sgn}{\mathrm{sgn}}
\newcommand{\tr}{\operatorname{Tr}}
\newcommand{\OmegaHat}{\widehat{\Omega}}
\newcommand{\Qhat}{\widehat{Q}}
\newcommand{\piHat}{\widehat{\pi}}
\renewcommand{\choose}[2]{\genfrac{(}{)}{0pt}{}{#1}{#2}}
\renewcommand{\S}{\mathcal{S}}
\newcommand{\Ccal}{\mathcal{C}}
\renewcommand{\epsilon}{\varepsilon}
\renewcommand{\phi}{\varphi}

\numberwithin{equation}{section}

\begin{document}

\setcounter{page}{1}

\title[The Swapping Algorithm on the Blume-Emery-Griffiths Model]{Mixing times for the Swapping Algorithm on the Blume-Emery-Griffiths Model}

\author[Mirko Ebbers]{Mirko Ebbers}
\address[Mirko Ebbers]{Fachbereich Mathematik und Informatik,
Universit\"at M\"unster,
Einsteinstra\ss e 62,
48149 M\"unster,
Germany}

\email[Mirko Ebbers]{mirkoebbers@uni-muenster.de}

\author[Holger Kn\"opfel]{Holger Kn\"opfel}
\address[Holger Kn\"opfel]{Fachbereich Mathematik und Informatik,
Universit\"at M\"unster,
Einsteinstra\ss e 62,
48149 M\"unster,
Germany}

\email[Holger Kn\"opfel]{holger.knoepfel@ruhr-uni-bochum.de}

\author[Matthias L\"owe]{Matthias L\"owe}
\address[Matthias L\"owe]{Fachbereich Mathematik und Informatik,
Universit\"at M\"unster,
Einsteinstra\ss e 62,
48149 M\"unster,
Germany}

\email[Matthias L\"owe]{maloewe@math.uni-muenster.de}

\author[Franck Vermet]{Franck Vermet}
\address[Franck Vermet]{Laboratoire de Math\'ematiques de Bretagne Atlantique, UMR CNRS 6205, Universit\'e de Bretagne Occidentale,  6, avenue Victor Le Gorgeu\\
CS 93837\\
F-29238 BREST Cedex 3\\
France}

\email[Franck Vermet]{Franck.Vermet@univ-brest.fr}


\date{\today}

\subjclass[2000]{Primary: 60 J 10 Secondary: 60 K 35}


\newcommand{\wlim}{\mathop{\hbox{\rm w-lim}}}
\newcommand{\na}{{\mathbb N}}
\newcommand{\re}{{\mathbb R}}

\newcommand{\vep}{\varepsilon}

\begin{abstract}
We analyze the so called Swapping Algorithm, a parallel version of the well-known Metropolis-Hastings algorithm,
on the mean-field version of the Blume-Emery-Griffiths model in statistical mechanics. This model has two parameters  and depending on their choice, the model exhibits either a first, or a second order phase transition. In agreement with a conjecture by Bhatnagar and Randall we find that the Swapping Algorithm mixes rapidly in presence of a second order phase transition, while becoming slow when the phase transition is first order.
\end{abstract}

\maketitle
\nocite{ZhengDissertation,BhatnagarRandallTorpidMixingPotts,EbbersLoweREM,EllisOttoTouchetteBEG,WoodardSchmidlerHuberRapidMixing,WoodardSchmidlerHuberTorpidMixing,MadrasZhengCW,GoreJerrumSW}
\section{Introduction}
Simulation methods are important tools in applied mathematics, e.g. in Bayesian statistics, computational physics, econometrics, or computational biology.
Markov Chain Monte Carlo (MCMC, for short) methods on the other hands belong to the most popular simulation techniques. They sample an unknown distribution, rely on the ergodic theorem for Markov chains, and construct a Markov chain on a finite state space that converges to the desired distribution. The first question is, of course, whether such a Markov chain exists. This is answered in the affirmative by the Metropolis-Hastings chain: Given an irreducible, aperiodic Markov chain (the base chain) on
the underlying state space, the Metropolis-Hastings algorithm allows to sample from a Markov chain with {\it any} given invariant distribution with full support. The idea of the Metropolis-Hastings algorithm, to always accept states with a higher probability than the current state and to accept states that are less likely with a probability equal to the ratio of the probability of the new state and the probability of the current state, is borrowed from the Glauber dynamics in statistical physics. In situations where the measure we want to sample from is a Gibbs distribution, as is often the case in statistical mechanics, the operation of comparing two probabilities can be performed quickly, i.e. with a small number of steps.

Like the Glauber dynamics the Metropolis-Hastings algorithm usually converges slowly, when the target distribution is multi-modal, i.e. when there are states that are locally very likely but globally not optimal.
Such situations occur e.g. in statistical physics in the presence of a phase transition and the slow convergence of the Glauber dynamics to the equilibrium distribution there is known under name of ''metastability''.

Several modifications of the Metropolis-Hastings algorithm have been proposed to circumvent this problem and speed up the convergence. Among them the so-called Swapping Algorithm (see \cite{GeyerMCMCmaximumLikelihood}), {also called Metropolis-coupled Markov chains or}  Parallel Tempering (see \cite{orlandini}), and the Simulated Tempering Algorithm (see \cite{marinari_parisi}, \cite{GeyerThompsonAMCMC}, and \cite{madras}) are very popular in applications, in particular on models from statistical physics.
In many situations they seem indeed to be able to improve the convergence of the Metropolis chain. However, the theoretical results about these
algorithms are rather limited: Madras and Zheng \cite{MadrasZhengCW} were able to show that the Swapping chain converges quickly for the Curie--Weiss model (among others). On the other hand, relying on results from Zheng's Ph.D. thesis (\cite {ZhengDissertation}), Bhatnagar and Randall \cite{BhatnagarRandallTorpidMixingPotts} prove that both, the Swapping Algorithm and Simulated Tempering, are slowly mixing for the 3-state Potts model and conjecture that this is caused by the first order phase transition in the Potts model (while the phase transition in the Curie-Weiss model is of second order). The techniques of these two papers were generalized to a couple of interesting situations by Huber, Schmidler, and Woodard, see \cite{WoodardSchmidlerHuberRapidMixing} and \cite{WoodardSchmidlerHuberTorpidMixing}.
A first rapid convergence result for the Swapping Algorithm in a disordered situation was proved by L\"owe and Vermet in \cite{loewe_vermet_swap}. Ebbers and L\"owe \cite{EbbersLoweREM} show that in disordered models the conjecture by Bhatnagar and Randall is not correct. They prove that the Swapping Algorithm mixes slowly on the Random Energy Model, even though this model has only a third order phase transition. This, however, may actually be a true disorder phenomenon, since in the theory of spin glasses, free energies are usually smoothed by taking expectations over the disorder.

The aim of the current paper therefore is to analyze the conjecture of Bhatnagar and Randall in another ordered model. A very appropriate scenario for this purpose is the mean-field version of the so called Blume-Emery-Griffiths (BEG, for short) model. This model
resembles a Curie-Weiss model with three states, $\pm 1$ and $0$. However, unlike  in the Potts model, the state $0$
plays  a particular role. The BEG
model has been studied extensively as a model of many diverse systems, in particular $He^3-He^4$ mixtures. A fact that makes it particularly interesting for our purposes is that, for different parameter values, it exhibits both, a discontinuous first-order phase transition and a continuous second order phase transition. This behavior has been conjectured for quite some time in the physics literature, but only recently was rigorously shown to be true in a paper by Ellis et al. \cite{EllisOttoTouchetteBEG}. One reason, why the mean-field version of the BEG model is mathematically challenging, is based on the fact, that even though the energy functions depends on a two dimensional parameter, the coordinates of this parameter are not independent. Other results on the BEG model were obtained by Ellis et al. in subsequent papers (\cite{Ellis2}, \cite{Ellis3}, \cite{Ellis4}), where the mean-field BEG model was referred to as mean-field Blume-Capel model. The Glauber dynamics for this model was studied in a
very recent paper by Kovchegov, Otto and Titus \cite{Otto}. They show that the mixing times of the Glauber dynamics undergoes a transition from rapid to slow mixing depending on the parameter values; the fascinating aspect of this result is, that the mixing time transition coincides with the equilibrium phase transition in the regime of the second order continuous phase transition but differs in the regime of the first-order discontinuous phase transition of the BEG model.

In the present paper, we consider the Swapping and Simulated Tempering Algorithms for the BEG model in regimes where the model is multimodal and confirm the conjecture by Bhatnagar and Randall in so far, that we are able to show rapid convergence (i.e. convergence in polynomial time in the system size) and torpid mixing (i.e. convergence in exponential time) depending on whether there is a second or a first order phase transition in the model.

As mentioned before, Woodard, Schmidler and Huber \cite{WoodardSchmidlerHuberRapidMixing} were able to give the first known result of rapid mixing of the Swapping Algorithm in a general, non model-specific, setting, in particular also to situations where the target distribution has more than one mode. We note that their result are so general, that they cannot be used in the case of rapid mixing in the BEG model. The technique used by Woodard, Schmidler and Huber relies heavily on a static, non temperature-dependent, partitioning of the state space. The underlying Metropolis chain needs to mix rapidly in each part, for any temperature, in order for their technique to work. Furthermore, the probability of each part must not get too small, as the temperature is decreased. In the rapid mixing case of the BEG model, this partitioning cannot be achieved. Our proof relies on a dynamic, temperature dependent, partitioning in which one part gets very unlikely as the temperature is decreased. More precisely, the  partitioning necessary for proofing rapid mixing as stated in Theorem \ref{Klarge} below
is given in our formula \eqref{Temperatur_dependant_partitioning_1} which uses the division of the state space for every temperature introduced in \eqref{Temperatur_dependant_partitioning_2} through \eqref{Temperatur_dependant_partitioning_3}. The necessity arises as there is no temperature independent partitioning such that the Metropolis chain itself is rapidly mixing for every partition and for every temperature. Additionally as defined in \eqref{Temperatur_dependant_partitioning_4} below it is even necessary to switch temperature dependently from one partition to two partitions (per total magnetization direction) at the critical temperature. This technique is indeed tailored for the bimodal situation of the BEG.

We organize the paper in the following way: The second section introduces the Swapping Algorithm (based on the Metropolis-Hastings chain) formally. At the same time we also introduce the Tempering Algorithm, which is itself uninteresting for applications in statistical mechanics, but provides a chain, that can be compared to the Swapping Algorithm, in particular when both algorithms are slowly mixing. {In Section 3 we introduce the  mean-field BEG model.  We propose a way to rewrite this model,   present a theorem  on the free energy  which  is a refinement of some results given in \cite{EllisOttoTouchetteBEG}, and is necessary for our analysis of the Swapping Algorithm. Then we give  our results  on the Swapping and Tempering Algorithms -- a characterization of the parameter regimes where these Algorithms converge rapidly or slowly, respectively. These results are proved in Section 4 and 5, respectively. The proofs use methods to bound the spectral gaps of Markov chains such as coupling methods or
Poincar\'e inequalities. In the appendices, we cite those bounds we need in the proofs. Moreover,
we prove a result on the speed of convergence of a coloring algorithm on a graph and our results on the free energy in the BEG model. These lemmata turn out to be useful in the proofs of our results in Section 4 and 5.

\section{Simulated Tempering and Swapping}
In this section we introduce two variants of the Metropolis-Hastings Algorithm. These algorithms include an additional change of temperature with the idea to speed up the Metropolis chain, when it is slow. They are specifically tailored for situations, where the invariant measure is a Gibbs measure with respect to some energy function and the Metropolis Algorithm mixes slowly at low temperatures, but quickly at high temperatures. We start with the Simulated Tempering Algorithm proposed by {Marinari and Parisi \cite{marinari_parisi}.}

\subsection{Simulated Tempering}
From now on and for the rest of the paper let us assume that the target distribution is a Gibbs measure on a finite set $\Omega$. To be more specific,
 let $H(\cdot)$ denote an energy function or Hamiltonian of the system. For every inverse temperature $\beta >0$, the probability function on $\Omega$ given by
\begin{equation}
 \pi_\beta(\sigma):=\frac{e^{\beta H(\sigma)}}{\sum_{\sigma'\in\Omega}e^{\beta H(\sigma')}}=\frac{e^{\beta H(\sigma)}}{Z(\beta)}\label{gibbs-measure}
\end{equation}
is called a Gibbs measure. Note that the sign of our energy function differs from the conventional choice in statistical mechanics. For the sake of this paper we will be concerned with simulating such Gibbs measures.

Let $\mathcal{K}_{gen}$ denote an aperiodic, symmetric and irreducible Markov chain on $\Omega$, the so-called base chain, and $T_\beta(\cdot, \cdot)$ the corresponding Metropolis-Hastings chain for $\pi_\beta$ defined by
\begin{equation} \label{Definition_Metropolis}
T_\beta(x,y)=\left\{\begin{array}{ll}
\mathcal{K}_{gen}(x,y) & \mbox{if } x\neq y \mbox{ and } H(y) \ge H(x) \\
\mathcal{K}_{gen}(x,y) \frac{\pi_\beta(y)}{\pi_\beta(x)} & \mbox{if } x\neq y \mbox{ and } H(y) < H(x) \\
1-\sum_{z\neq x} T_\beta(x,z) & \mbox{otherwise.}
\end{array}
\right.
\end{equation}

{For Gibbs measures on a finite set with some sort of neighborhood structure, one commonly chooses $\mathcal{K}_{gen}$ as a local random walk kernel.}
This algorithm, despite of being natural, is sometimes slow in natural situations, e.g. when sampling from the low temperature distribution of the Curie-Weiss model (see e.g. \cite{MadrasPiccioni}).
To speed up its convergence, we consider
$\Omega \times \{0,1,...,M\}$ for some $M\in\N$. In the case of  Gibbs measures on a set $\Omega= {\mathcal S}^N$, where $\mathcal S$ is some set with more than one element and $N$ is a large natural numvber, $M$  will be typically chosen as $M:=c_1 N$ for some constant $c_1>0$. The second component of the new state space refers to the current temperature of the model (or the chain, resp.). Define
\begin{equation}
\beta_i:=\frac{i}{M}\beta \mbox{  and the probability measures } \pi_i:=\pi_{\beta_i}.
\end{equation}
As probability measure on $\Omega\times\{0,...,M\}$ we take
\begin{equation}
\pi(x)=\pi((\sigma,i))=\frac{1}{M+1}\pi_i(\sigma).
 \end{equation}
 We construct a Markov chain that starts in $(\sigma,i)\in\Omega\times\{0,1,...,M\}$ and chooses a new state $(\sigma',i)$ according to $T_{\beta_i}$. In a second step the temperature is changed according to a similar Metropolis chain. The idea is, that in case the chain is in an energy-valley, it can increase its temperature (reduce $\beta$) and thereby reduce the cost of switching to another energy-valley. Explicitly, this works as follows:

In the first step let $i\in\{0,...,M\}$ be fixed. Then a transition from $(\sigma,i)$ to $(\sigma',i)$ has probability $P_{st}((\sigma,i),(\sigma',i)):= T_{\beta_i}(\sigma,\sigma')$.
In the second step let $\sigma\in\Omega$ be fixed. Then the chain moves from $(\sigma,i)$ to $(\sigma,j)$ according to the transition probabilities
\begin{displaymath}
Q((\sigma,i),(\sigma,j)):=\left\{\begin{array}{l l}
                 K_{\text{tm}}(i,j) & \text{if } \pi_j(\sigma)\geq \pi_i(\sigma) \text{ and } i\neq j\\
         K_{\text{tm}}(i,j)\frac{\pi_j(\sigma)}{\pi_i(\sigma)} & \text{if } \pi_j(\sigma) < \pi_i(\sigma)\\
         1-\sum\limits_{k\neq i} Q((\sigma,i),(\sigma,k)) & \text{if } i=j\\
                \end{array}\right.
\end{displaymath}
with
\begin{displaymath}
 K_{\text{tm}}(i,j):=\left\{\begin{array}{l l}
                             \frac{1}{2(M+1)} & \text{if } j=i\pm 1 \mbox{ and } j \in \{0,\ldots, M\}\\
                 0 & \text{if } |i-j|>1\\
                 1-\sum\limits_{k\neq i} K_{\text{tm}}(i,k) & \text{if } i=j.\\
                            \end{array}\right.
\end{displaymath}

The actual Simulated Tempering Algorithm now consists of first applying a temperature move $Q$, then a Metropolis move at the present temperature (the transition matrix of which is denoted by $T$), and finally another temperature move. Hence, in terms of transition matrices the Simulated Tempering algorithm is given by
$ QP_{st}Q.$

Notice that the computation of $\frac{\pi_j(\sigma)}{\pi_i(\sigma)}$ in the matrix $Q$ needs knowledge of the normalizing constants $Z(\beta_i)$ and $Z(\beta_j)$ which in most cases is hard to obtain. This is the reason for introducing the following Swapping Algorithm.

\subsection{Swapping}
The so called Swapping Algorithm was suggested by Geyer in \cite{GeyerMCMCmaximumLikelihood}. The basic idea of changing the temperature is maintained. As state space for the
Swapping chain we choose:
$$\OmegaSW:=\Omega^{M+1}$$
A natural choice for a probability measure on $\OmegaSW$ is:
\begin{equation}
\pi(x):=\prod\limits_{i=0}^{M} \pi_i(x_i)=\frac{\prod\limits_{i=0}^{M} e^{\frac{i\beta}{M} H(x_i)}}{\prod\limits_{i=0}^{M} Z(\beta_i)}
\end{equation}
with $x=(x_0,...,x_M)\in\OmegaSW$. As in the Simulated Tempering Algorithm the Swapping Algorithm consists of two steps. In the first step, we choose an $i\in\{0,...,M\}$ uniformly and update the $i$-th component of the current state $x=(x_0,...,x_M)$ according to the usual Metropolis chain $T_{\beta_i}$ at inverse temperature $\beta_i$. In the second step we choose an $i\in\{0,...,M-1\}$ uniformly at random and swap the components $x_i$ and $x_{i+1}$ of $x$ with probability
$$\min\left(1,\frac{\pi(x_0,...,x_{i+1},x_i,...,x_M)}{\pi(x_0,...,x_i,x_{i+1},...,x_M)}\right).$$

So explicitly the first step works as follows: The transition probabilities from $x=(x_0,...,x_{i-1},x_i,x_{i+1},...,x_M)\in\OmegaSW$ to $x'=(x_0,...,x_{i-1},x_i',x_{i+1},...,x_M)$ are $T_i(x,x'):=T_{\beta_i}(x_i, x_i')$.
For any  $u, v$, let $\delta(u,v)=1$ if $u=v$ and $0$ otherwise.
Then the product chain
\begin{multline}\label{defP}
 P(x,y)=\frac{1}{2}\delta(x,y)+\frac{1}{2(M+1)}\sum\limits_{i=0}^{M} \delta(x_0,y_0)\cdot \dots \cdot \delta(x_{i-1},y_{i-1})T_i(x_i,y_i)\\ \times \delta(x_{i+1},y_{i+1})\cdot ... \cdot \delta(x_M,y_M)
\end{multline}
gives us a Markov chain on $\OmegaSW$. Also note that we never change more than one component at a time.
The second step is the temperature swap. Here the transition probabilities from $x=(x_0,...,x_i,x_{i+1},...,x_M)$ to $x'=(x_0,...,x_{i+1},$ $x_i,...,x_M)$ are
\newcommand{\Ksw}{K_\text{sw}}
\begin{displaymath}
Q(x,x'):=\left\{\begin{array}{l l}
                 \Ksw(x,x') & \text{if } \pi(x')\geq \pi(x) \text{ and } x\neq x'\\
         \Ksw(x,x')\frac{\pi(x')}{\pi(x)} & \text{if } \pi(x') < \pi(x)\\
         1-\sum\limits_{z\neq x} Q(x,z) & \text{if } x=x'.\\
                \end{array}\right.
\end{displaymath}
$\Ksw$ is defined by
\begin{displaymath}
 \Ksw(x,x'):=\left\{\begin{array}{l l}
                     \frac{1}{2M} & \text{if } \exists i \text{ with } x_j=x_j' \text{ } \forall j\notin\{i,i+1\}, \\&\text{and } x_i=x_{i+1}', x_{i+1}=x_i'\\
             0 & \text{if } \nexists i \text{ with } x_j=x_j' \text{ } \forall j\notin\{i,i+1\},\\&\text{and } x_i=x_{i+1}', x_{i+1}=x_i'\\
             1-\sum\limits_{z\neq x} \Ksw(x,z) & \text{if } x=x'\\
                     \end{array}\right..
\end{displaymath}
Note that the factor $\frac{1}{2}$ in the definition of $\Ksw$ and $P$ guarantees that both, $P$ and $Q$, are aperiodic and that the corresponding operators are positive. Notice that all the normalizing constants in $Q$ and $P$ cancel out, such that the transition probabilities can be effectively computed.

The Swapping Algorithm is now any reasonable combinations of $P$ and $Q$, usually one takes $QPQ$ as it is reversible with respect to $\pi$ if $Q$ and $P$ are reversible (which in our situation is the case). The following theorem gives an idea, how the speed of convergence of swapping and tempering depend on each other.
\begin{theorem}[Zheng \cite{Zheng2003}]\label{vergleich}
If there exists a constant $\delta>0$ such that
$$\sum\limits_{x\in\Omega} \min\{\pi_i(x),\pi_{i+1}(x)\}\geq \delta\quad \text{for all } 1\leq i\leq M,$$
then if the Swapping Algorithm converges in polynomial time, so does the Simulated Tempering Algorithm.
\end{theorem}

\section{Results}
\label{section_Results}
Now we introduce the mean field Blume-Emery-Griffiths (BEG) model. For a given $K>0$ the Hamilton function on $\Omega=\{-1,0,1\}^{N}$ is given by
\begin{equation}\label{BEG_Hamiltonian}
 H(\sigma)=H_K(\sigma):=- \sum_{j=1}^{N} \sigma_j^2 + \frac{K}{N}\left(\sum_{j=1}^{N}\sigma_j\right)^2
\end{equation}
for $\sigma\in\Omega$. Here a state $\sigma$ is said to have spin $\sigma_i$ in coordinate $i$. Therefore the Gibbs measure of the BEG model, which we want to sample from, is
\begin{equation}\label{BEGpi}
 \pi_\beta(\sigma)=\frac{e^{\beta H(\sigma)}}{Z(\beta)}=\frac{e^{\beta H(\sigma)}}{\sum_{\sigma'}e^{\beta H(\sigma')} }=
 \frac{e^{\beta (- \sum_{j=1}^{N} \sigma_j^2 + \frac{K}{N}(\sum_{j=1}^{N}\sigma_j)^2)}}{\sum_{\sigma'}e^{\beta H(\sigma')}}
\end{equation}
with $Z(\beta)$ being the normalization constant.
We see, that in the mean-field BEG model, the energy function solely depends on the parameters $\sum_{j=1}^{N} \sigma_j^2$, and $(\sum_{j=1}^{N} \sigma_j)^2,$ the last one being the term of interactions between spins. It can therefore be expected, that the mean-field BEG can be rigorously analyzed. However, as the two parameters are strongly dependent, the analysis is not easy. It was not until the paper by Ellis et al. \cite{EllisOttoTouchetteBEG}, that one obtained a thorough understanding of
the macroscopic behavior of the mean field BEG model. In a nutshell their result coincides with an intuitive understanding of the model. If $K$ is large enough, the second term becomes dominant and the model behaves like the Curie-Weiss model (see \cite{Elli1985} for an analysis of the latter model): it has a second order phase transition at some critical temperature $\beta_c^{(2)}(K)$. When $K$ becomes smaller, this phase transition however is of first order, the low temperature macro-states emerge discontinuously from the high-temperature macro-state. If $K$ is eventually too small, there is no phase transition at all.

We will first do some system specific preparations, in order to get more familiar with the model.
To simplify notation define the functions
\begin{align}
 S_N(\sigma)&=\sum_{i=1}^{N} \sigma_i\\
 R_N(\sigma)&=\sum_{i=1}^{N} \sigma_i^2
\end{align}
where $S_N$ gives the total magnetization, and $R_N$ the total number of non-zero spins of the state $\sigma$. Using this notation we define
\begin{equation}
\A_{s,r}:=\big\{\sigma\in\Omega\big|S_N(\sigma)=s, R_N(\sigma)=r\big\}\label{A_sr}
\end{equation}
as the set of states with a fixed number of $0$s and fixed magnetization. As we consider the mean-field BEG model, all states in $\A_{s,r}$ are basically indistinguishable in the system. We will later (see Theorem \ref{Theorem_T^2_fast_on_A_sr} below) see, that the Metropolis chain $T^2$ restricted to $\A_{s,r}$ mixes rapidly for any combination of $s$ and $r$.

In order to be able to better address non-negligible differences in the state space consider
\begin{equation}
 \Upsilon=\Upsilon_N:=\big\{\textbf{a}=(a_{-1},a_0,a_1)\in\R^3\big|a_i\geq 0\ \forall i, \sum_i a_i=1, Na_i\in\N\ \forall i=-1,0,1\big\}
\end{equation}
such that
\begin{equation}\label{Equation_BEG_IdeaByGoreAndJerrum}
 \Omega=\bigcup_{\textbf{a}\in\Upsilon} \Big\{\sigma\in\Omega\Big|\sum_{j=1}^{N} \delta(\sigma_j,i)=Na_i\ \forall i\in\{-1,0,1\}\Big\}
\end{equation}
is a disjoint union. Note that all states in one of the sets on the right hand side of \eqref{Equation_BEG_IdeaByGoreAndJerrum} only differ by an index permutation and thereby have the same energy. This is inspired by Gore's and Jerrum's work on the Potts Model \cite{GoreJerrumSW} as the following calculation makes the state space easier to handle.

Considering
\begin{align}
 \pi_\beta(\sigma\text{ has type }N\textbf{a})&=\choose{N}{Na_{-1},Na_0,Na_1} Z(\beta)^{-1} e^{-\beta\big(Na_{-1}+Na_1-\frac{K}{N}(Na_1-Na_{-1})^2\big)}\notag\\
	&=\choose{N}{Na_{-1},Na_0,Na_1} Z(\beta)^{-1} e^{-N\beta\big(a_{-1}+a_1-K(a_1-a_{-1})^2\big)}\label{Eq4.6}
\end{align}
and using Stirling's approximation one obtains
\begin{align}
 \pi_\beta(\sigma\text{ has type }N\textbf{a})&=Z(\beta)^{-1} N^{-1} e^{-N\big(\sum_i a_i\log a_i\big) + \Delta(\textbf{a})} e^{-N\beta\big(a_{-1}+a_1-K(a_1-a_{-1})^2\big)}\notag\\
&=Z(\beta)^{-1} N^{-1} e^{N\big(\beta\big(-a_{-1}-a_1+K(a_1-a_{-1})^2\big)- \sum_i a_i\log a_i\big) + \Delta(\textbf{a})}\label{EnergyLandscape_Inducing-P_f}
\end{align}
with $|\Delta(\textbf{a})|=O(1)$ if there exists an $\epsilon>0$ with $a_i\geq\epsilon$ for all $i\in\{-1,0,1\}$. So understanding
\begin{equation}\label{EnergyLandscape}
 f_\beta(\textbf{a}):=\beta\big(-a_{-1}-a_1+K(a_1-a_{-1})^2\big)-\sum_i a_i\log a_i
\end{equation}
will give us a better insight in how the BEG model behaves as a function of $\beta$.
{
First, we prove the following result in the appendix:

\begin{theorem}\label{Theorem_Holger}
$f_\beta$ has at most three local maxima on $\Upsilon_\infty:=\{(a_{-1}, a_0, a_1)\in \R_+^3: \sum_{i=-1}^1 a_i=1\}$. There are no further maxima on the boundary of $\Upsilon_\infty$.
\end{theorem}
}


Moreover, in \cite{EllisOttoTouchetteBEG} (Sections 3 and 5) one finds a complete description of the set ${\mathcal E}_{\beta,K}$ of  the maxima of $f_\beta$ on $\Upsilon_\infty$, i.e. the set of canonical equilibrium macro-states of the model, for all $\beta$ and $K$. We will adopt the notation of  \cite{EllisOttoTouchetteBEG} for the critical values of the parameters $\beta$ and $K$: there exists a critical value  $\beta_c= \log 4$, such that ${\mathcal E}_{\beta,K}$  has two different forms for $0< \beta \leq \beta_c$ and for $\beta >\beta_c$. More precisely, for $0< \beta \leq \beta_c$, there exists a critical value $K^{(2)}_c(\beta)=\frac 1{4\beta e^{-\beta}}+ \frac 1{2\beta}$, such that ${\mathcal E}_{\beta,K}$ is unimodal for $0< K< K^{(2)}_c(\beta)$, and bimodal for $K> K^{(2)}_c(\beta)$. Moreover, ${\mathcal E}_{\beta,K}$ exhibits a continuous bifurcation at $K^{(2)}_c(\beta)$, which corresponds to a second order phase transition.

For $\beta > \beta_c$, there exists a critical value $K^{(1)}_c(\beta)$ such that ${\mathcal E}_{\beta,K}$ is unimodal for $0< K< K^{(1)}_c(\beta)$, trimodal for $K= K^{(1)}_c(\beta)$ and bimodal for $K> K^{(1)}_c(\beta)$. Moreover, ${\mathcal E}_{\beta,K}$ exhibits a discontinuous bifurcation at $K^{(1)}_c(\beta)$, which corresponds to a first-order phase transition. The quantity $K^{(1)}_c(\beta)$ is defined  implicitly  in  \cite{EllisOttoTouchetteBEG}, but an explicit form is not obtained. This is consistent with the general challenge in analyzing first-order, discontinuous phase transitions in statistical physics models. As a consequence, to study the behavior of  $K^{(1)}_c(\beta)$  as ${\beta\rightarrow + \infty}$ is not trivial.  We prove in the appendix \ref{existenceKlow} that  this limit  exists, and Ellis et al.  \cite{EllisOttoTouchetteBEG} indicate that numerical simulations lead to the conjecture that $\Klow :=\lim_{\beta\rightarrow + \infty} K^{(1)}_c(\beta)$ is equal to 1.

A slight difficulty of the above discussion is also that the conventional picture of statistical mechanics where one studies a model depending on temperature is turned upside down: The critical parameters are defined as function of $\beta$ and not the other way round.

In Section 5 of \cite{EllisOttoTouchetteBEG}  the authors extrapolate these results obtained by fixing $\beta$ and varying $K$ to results about the phase transition behavior of the canonical equilibrium macro-states for fixed $K$ and varying $\beta$.
We define the tricritical value $K_c= K^{(2)}_c(\beta_c)\simeq 1.0820$.
Then for $K> K_c$, there exists a value $\beta_c^{(2)}(K)$ such that ${\mathcal E}_{\beta,K}$ exhibits a second order phase transition at $\beta =\beta_c^{(2)}(K)$:  there exists a $\delta>0$, such that ${\mathcal E}_{\beta,K}$ exhibits a single phase for $\beta \in (\beta_c^{(2)}(K) - \delta, \beta_c^{(2)}(K)]$ and two distinct phases for $(\beta_c^{(2)}(K), \beta_c^{(2)}(K)+ \delta)$.
 And for $\Klow<K< K_c$ (we precise in Corollary \ref{inverse} why we need the condition $K>\Klow$), there exists a value $\beta_c^{(1)}(K)$ such that ${\mathcal E}_{\beta,K}$ exhibits a first-order phase transition at $\beta =\beta_c^{(1)}(K)$:  there exists a $\delta>0$, such that ${\mathcal E}_{\beta,K}$ exhibits a single phase for $\beta \in (\beta_c^{(1)}(K) - \delta, \beta_c^{(1)}(K))$, three distinct phase at $\beta = \beta_c^{(1)} (K)$, and two distinct phases for $(\beta_c^{(1)}(K), \beta_c^{(1)}(K)+ \delta)$.

\vspace{2mm}
These properties imply in particular that  the  Metropolis algorithm is torpidly mixing for the BEG model, for the values  of $ (\beta,K)$ such that the model is multimodal,  if the base chain is a local random walk kernel.
In fact,
we know that $\pi_\beta(\sigma\text{ has type }N\textbf{a})$ has exponential structure. We also know that for suitable $K, f_\beta$ has at least two  modes for sufficiently (depending on $K$) large $\beta$. Take $\textbf{a}$ to represent one of the   maximum point. If we define $B_\epsilon(\textbf{a})$ as the ball of radius $\epsilon$ centered in $\textbf{a}$ in the appropriate metric space, this leads to $B_\epsilon(\textbf{a})$ having exponential little conductance, therefore representing a bad cut in the state space. For more details see our Section \ref{Subsection-Case_Klow<K<Kc} where this technique is used in the more complicated setup of swapping.

\vspace{2mm}
In the present paper,  we will consider the  Simulated  Tempering Algorithm  and  the Swapping Algorithm, which are defined in Section 2, for  values of  $(\beta, K)$ such that the Metropolis algorithm is torpidly mixing for the BEG model. We will focus on two regions of the parameters  $(\beta, K)$ where we show the influence of the order of the phase transition on the speed of convergence of both algorithms.
For the Simulated  Tempering Algorithm  and  the Swapping Algorithm, the corresponding Metropolis-Hastings chain for the measure $\pi_{\beta}$, defined in  (\ref{BEGpi}), is given by (\ref{Definition_Metropolis}), with the proposal chain
$$\mathcal{K}_{gen}(x,y)= \frac 1 {4N},$$ if $x,y\in\{-1,0,1\}^N$ and differ in exactly one spin $x_i\ne y_i,$ for some $i\in\{1,...,N\}$, and $\mathcal{K}_{gen}(x,x)= \frac 12$. In all other cases define
$$\mathcal{K}_{gen}(x,y)= 0.$$

The BEG Model, as Ellis et al. \cite{EllisOttoTouchetteBEG} show, exhibits different phase behavior depending on $K$. For small $K<\Klow$ there is, for every temperature, only one macro state, which implies that there is no phase transition.

The first regime we want to look at is $\Klow<K<K_c$ with $\Klow :=\lim_{\beta\rightarrow + \infty} K^{(1)}_c(\beta)$ and $K_c=K(\log 4)$ as in \cite[Eq. (3.19)]{EllisOttoTouchetteBEG}. The model exhibits a discontinuous phase transition at a $\beta_c^{(1)}(K)$ depending on $K$. We will use this discontinuity in the phase to show

\begin{theorem}\label{Kintermediate}
 Consider the BEG model with $\Klow< K< K_c$. Then for $\beta>\beta_c^{(1)}(K)$, the Simulated Tempering Algorithm is torpidly mixing, since
$$\Gap(Q P_{st} Q)\leq e^{-cN}$$
holds for $c>0$ as constructed in Theorem \ref{Metropolis_Low_Conductance}.
\end{theorem}

We prove this theorem in Section \ref{Subsection-Case_Klow<K<Kc}.
\begin{corollary}
 This implies torpid mixing of the Swapping Algorithm in this regime.
\end{corollary}
For $K>K_c$ the model shows a continuous phase transition at $\beta_c^{(2)}(K)$ which will lead to a Swapping chain which behaves  like a Curie-Weiss model's Swapping chain which Madras and Zheng already considered in \cite{MadrasZhengCW}.
{However, the technique used by Madras and Zheng relies  on a static, non temperature-dependent, partitioning of the state space. The underlying Metropolis chain needs to mix rapidly in each part, for any temperature. In the rapid mixing case of the BEG model, this partitioning cannot be achieved. Our proof relies on a dynamic, temperature dependent, partitioning in which one part gets very unlikely as the temperature is decreased. For the BEG model, the proof becomes much more involved, but we can use ideas of  Madras and Zheng  \cite{MadrasZhengCW} and (a corrected version of ideas in) Bhatnagar and Randall \cite{BhatnagarRandallTorpidMixingPotts}  to get}

\begin{theorem}\label{Klarge}
 For $K>K_c$  and $\beta>\beta_c^{(2)}(K)$, the Swapping chain with its transition kernel $QPQ$ for the BEG model is rapidly mixing, since
$$\Gap(QPQ)\geq \frac{1}{p(N)}$$
for some polynomial $p$ of $N$.
\end{theorem}

We prove this theorem in Section \ref{Section_Proofs}.

\begin{remark}
Giving an explicit bound would need a longer argument in the end of the proof of Theorem \ref{Rapid_mixing_of_RW1} which does not give a better insight of the situation. As we do not believe our technique to give a sharp bound anyway,  we refrain from doing this extra step and do not give a suitable polynomial explicitly.
\end{remark}

\begin{corollary}
 This implies rapid mixing of the Simulated Tempering chain $Q\Pst Q$ in this regime.
\end{corollary}

\section{Proof of Theorem \ref{Klarge} }\label{Section_Proofs}

\subsection{General partitioning of the state space in the case of $K>K_c$}
We will begin to show Theorem \ref{Klarge} by partitioning the state space
\begin{equation}\Omega=\{-1,0,1\}^N=\OmegaP\cup\OmegaM\end{equation}
into two disjoint almost equally large parts
\begin{align*}
\OmegaP&=\Big\{\sigma\in\Omega\Big|\sum_i \sigma_i> 0\Big\}\cup\{(0,...,0)\}\\
    &\quad\cup\Big\{\sigma\neq (0,...,0)\Big| \sum_i \sigma_i =0, \text{ with the first non-zero coordinate =+1}\Big\}\\
\OmegaM&=\{\sigma\in\Omega|\sum \sigma_i< 0\}\\
    &\quad\cup\Big\{\sigma\neq (0,...,0)\Big| \sum_i \sigma_i =0, \text{ with the first non-zero coordinate =-1}\Big\}.
\end{align*}
Using this partitioning we will decompose $\OmegaSW=\Omega^{M+1}$ in the same way as Madras and Zheng in \cite[Section 4, Step two]{MadrasZhengCW}.

Let $\widetilde\OmegaSW:=\{+,-\}^M$ and take $x\in\OmegaSW$. Define the signature of $x$ by
\begin{equation}\begin{array}{r c c c l}
                 \sgn&:&\OmegaSW&\to&\widetilde\OmegaSW\\
		&&x&\mapsto&v
                \end{array}\end{equation}
with
\begin{equation} v_i=\begin{cases}
                      	+&\text{if }x_{i+1}\in\OmegaP\\
			-&\text{if }x_{i+1}\in\OmegaM,
                     \end{cases}\end{equation}
such that $\sgn(x)$ contains the sign, of the total magnetization of each component of $x$ except of the component for $\beta=0$. The first component of $x$ will have a special role, which will become apparent within the next paragraphs.

We will decompose the state space using the number of $+$-signs in $\sgn(x)$. For fixed $k\in\{0,...,M\}$ define
\begin{equation}
 \widetilde\Omega_k:=\{v\in\widetilde\OmegaSW|v \text{ has exactly } k\ +\text{-signs}\}.
\end{equation}
and note, that
$$\OmegaSW=\bigcup_{k=0}^{M} \Omega_k$$
is a disjoint union of
\begin{equation}\Omega_k:=\{x\in\OmegaSW|\sgn(x)\in\widetilde\Omega_k\}.\label{Definition-Omega-k}\end{equation}
Define $\overline Q$ to be the aggregated transition matrix and $(QPQ)|_{\Omega_k}$ to be the restriction of the chain $QPQ$ to the set $\Omega_k$ as defined in Theorem \ref{CaraccioloPelissettoSokal} for this decomposition.
Using Lemma \ref{MadrasZhengLemma7} and Theorem \ref{CaraccioloPelissettoSokal} we get
\begin{equation}
 \Gap(QPQ)\geq\Gap(Q^{\frac12}(QPQ)Q^{\frac12})\geq \Gap(\overline Q)\cdot \min_{k\in\{0,...,M\}} \Gap((QPQ)|_{\Omega_k}).
\end{equation}
Citing \cite[Sec. 4, step three]{MadrasZhengCW}, we can do all displayed calculations in our setting as well, which eventually leads to
\begin{equation}
 \Gap(Q^{\frac12}(QPQ)Q^{\frac12})\geq \frac18 \Gap(\overline Q)\cdot \min_{k\in\{0,...,M\}} \Gap((Q_k P_k Q_k))\label{Equation_QPQ-RestrictedTo_k}
\end{equation}
with $P_k$ and $Q_k$ being the restrictions of $P$ and $Q$ to $\Omega_k$, respectively (for a definition see Theorem \ref{CaraccioloPelissettoSokal} in the appendix).

The transition kernel $\overline Q$ is, in this setting, responsible for changing the number of components in $x\in\OmegaSW$ which are in $\OmegaP$ and $\OmegaM$, respectively. $\overline Q$ is essentially a one dimensional nearest neighbor random walk on $\{0,...,M\}$ whose spectral gap is well understood. Due to the symmetry in the model it does not (noticeably) matter for the chain, whether we restrict a given component $k$ of $x$ to be in $\OmegaP$ or $\OmegaM$.
This leads to
\begin{equation}
 \Gap((Q_k P_k Q_k))\approx\Gap((Q_{k'} P_{k'} Q_{k'}))\quad \forall k,k'\in\{0,...,M\}
\end{equation}
where $\approx$ means that both spectral gaps are of the same (polynomial or exponential) order.
This in turn implies $\min_{k\in\{0,...,M\}} \Gap((Q_k P_k Q_k)) \approx \Gap((Q_M P_M Q_M))$. We will write this as
\begin{equation} \min_{k\in\{0,...,M\}} \Gap((Q_k P_k Q_k))\approx\Gap((Q_M P_M Q_M))=\Gap((Q P_+ Q)),\end{equation}
where by abuse of notation, $Q_M$ is denoted by $Q$ and $P_M$ by $P_+$.
Note also that all arguments of the proof work in exactly the same way for any $k\in\{0,...,M\}$. The only difference is, which part of the state space we look at, for a given temperature $\beta_i$. The quantities
$\Gap(\overline Q)$ and $\Gap(Q P_+ Q)$ will be bounded below in the following subsections \ref{Speed1} and \ref{Speed2}.

\subsection{Speed of convergence of $\overline Q$}\label{Speed1}
Following in principle the proof given in \cite[Section 5]{MadrasZhengCW} (also see \cite[Section 2.5]{ZhengDissertation} for more details) we gain
\begin{lemma}
The spectral gap of the aggregated chain $\overline Q$ satisfies
$$\Gap(\overline Q)\geq \frac{1}{4 M^2} e^{-\beta(K+1)\frac{N}{M}}.$$
\end{lemma}
\begin{remark}
 Remark that the for the number of spins $N$ and the number of temperatures $M$ considered are interchanged between this paper and the reference given above. On the other hand, the notation now agrees with the standard notation in statistical mechanics.
\end{remark}
\begin{proof}
We first verify that the probability for an accepted swapping move is bounded below by a constant.
Using the notation given in \cite{MadrasZhengCW} let us define
$$
\rho_{i,i+1} :=\min\left(1, \frac{\pi_i(x_{i+1}) \pi_{i+1}(x_i)}{\pi_i(x_{i}) \pi_{i+1}(x_{i+1})}\right).
$$
Then
\begin{align}
\rho_{i,i+1}&=\min\left(1,\frac{e^{\beta_{i+1} H(x_i)}e^{\beta_{i} H(x_{i+1})}}{e^{\beta_{i} H(x_i)}e^{\beta_{i+1} H(x_{i+1})}}\right)\notag\\
  &=\min\left(1,e^{\beta_{i+1} H(x_i)+\beta_{i} H(x_{i+1})-\beta_{i} H(x_i)-\beta_{i+1} H(x_{i+1})}\right)\notag\\
  &=\min\left(1,e^{\beta \frac{i+1}{M} H(x_i)+\beta \frac{i}{M} H(x_{i+1})-\beta \frac{i}{M} H(x_i)-\beta \frac{i+1}{M} H(x_{i+1})}\right)\notag\\
  &=\min\left(1,e^{\beta \frac{H(x_i)}{M}-\beta \frac{H(x_{i+1})}{M}}\right)\notag\\
  &\geq e^{-\beta \frac{H(x_{i+1})}{M}}\notag\\
  &\geq e^{-\beta \frac{N(K+1)}{M}}\label{Lemma61_1}
\end{align}
as $H\leq (K+1)N$ implies \eqref{Lemma61_1} to be true.

Due to the definition of $\OmegaP$ and $\OmegaM$ it is clear, that $\pi_\beta(\OmegaP)=\frac12(1+1/Z_\beta)$ for any $\beta\geq0$. Recalling equations \eqref{Eq4.6} and \eqref{EnergyLandscape} and Theorem \ref{Theorem_Holger}  it is possible to find for any $\beta>0$ constants $0<c_1<c_2$ such that $Z_{\beta'}\in[e^{c_1 N},e^{c_2 N}]$ for all $\beta'\in[0,\beta]$. Using
\begin{equation}\label{Lemma61-EqRef1}
 1\leq \big(1+e^{-cN}\big)^M
  \leq e^{e^{-cN}M}
  \goesto 1
\end{equation}
as $N \to \infty$, we gain a constant $a>1$ such that for all sufficiently large $N$ and any $\nu\in\{-,+\}^{M}$
$$\pi(\Omega\times\Omega_{\nu_1}\times\cdots\times\Omega_{\nu_M})\in 2^{-M} [a^{-1},a]$$
holds. Recalling the definition of $\Omega_k$ in \eqref{Definition-Omega-k} we conclude
\begin{equation}
\pi(\Omega_k)=\sum_{\nu\in\widetilde\Omega_k} \pi(\Omega\times\Omega_{\nu_1}\times\cdots\times\Omega_{\nu_M})\in \choose{M}{k}\left(\frac{1}{2}\right)^{M} [a^{-1},a].
\end{equation}
As we want to use Lemma \ref{Lemma_Compare_Derichlet_forms} later on, in order to compare $\overline Q$ to an easier Markov chain, it is of interest to study the quantity
\begin{equation}
\pi(\Omega_i)\overline Q(i,i+1).
\end{equation}
Consider an $x\in\Omega_i$ and $y\in\Omega_j$. In case $|j-i|>1$ it is obviously impossible for the pure Swapping chain $Q$ to accept a step from $x$ to $y$, thus:
$$Q(x,y)=0, \text{ if } x\in\Omega_i, y\in\Omega_j\text{ with } |i-j|>1.$$
Hence,
$$\overline Q(i,j)=0, \text{ if }|i-j|>1.$$
The only way $i$ can change is by interchanging the first two coordinates $x_0$ and $x_1$ of $x$. For $0\leq i<N$,
we obtain
\begin{align}
\pi(\Omega_i)\overline Q(i,i+1)&=\sum_{x\in\Omega_i}\sum_{y\in\Omega_{i+1}}\pi(x)Q(x,y)\notag\\
  &=\sum_{x_0\in\Omega_+}\sum_{x_1\in\Omega_-}\sum_{\substack{x'\in\Omega_i\\ x_0'=x_0, x_1'=x_1}} \pi(x') Q(x',(0,1)x')\notag\\
  &=\sum_{x_0\in\Omega_+}\sum_{x_1\in\Omega_-}\sum_{\substack{x'\in\Omega_i\\ x_0'=x_0, x_1'=x_1}} \pi(x') \frac{1}{2M} \rho_{0,1}(x_0,x_1)\notag\\
  &=\frac{1}{2M} \sum_{x_0\in\Omega_+}\sum_{x_1\in\Omega_-} \pi_0(x_0)\pi_1(x_1)\rho_{0,1}(x_0,x_1) \sum_{\substack{x'\in\Omega_i\\ x_0'=x_0, x_1'=x_1}} \prod_{j=2}^{M}\pi_j(x'_j)\notag\\
  &\in \frac{1}{2M} \sum_{x_0\in\Omega_+}\sum_{x_1\in\Omega_-} \pi_0(x_0)\pi_1(x_1)\sum_{\substack{x'\in\Omega_i\\ x_0'=x_0, x_1'=x_1}} \prod_{j=2}^{M}\pi_j(x'_j) \left[e^{-\beta \frac{N(K+1)}{M}},1\right] \notag\\
  &\subseteq \frac{1}{2M} \choose{M-1}{i} \frac{1}{2^{M+1}} \left[ e^{-\beta \frac{N(K+1)}{M}} a^{-1},a \right]\notag
\end{align}
with the natural definitions of the sets in the last two lines.

We will now give another, much simpler, Markov chain whose spectral gap has been intensively studied. Consider the symmetric random walk $S$ on $\{0,...,M\}$, i.e.
\begin{align}
 S(0,1)=S(0,0)=S(M,M-1)&=S(M,M)\notag\\
  =S(i,i-1)=S(i,i+1)&=\frac12\text{ for }0<i<N.\notag
\end{align}
Let $r(i)=\choose{M}{i}2^{-M}$ be the binomial distribution on $\{0,...,M\}$, and let $R$ denote the Metropolis chain with proposal chain $S$ and reversible distribution $r(i)$. As has been shown by Diaconis and Saloff-Coste \cite[pp 698 and 719]{DiaconisSaloffCosteSobolev} $R$ satisfies
\begin{equation}\frac{1}{M}\leq\Gap(R)\leq\frac{2}{M}.\end{equation}
In order to use Lemma \ref{Lemma_Compare_Derichlet_forms} in the appendix first note that
\begin{equation}
 \pi(\Omega_i)\in\frac{1}{2^M}\choose{M}{i}[a^{-1},a]=r(i)[a^{-1},a]
\end{equation}
implies $r(i)\geq \frac 1 a\pi(\Omega_i)$ for all $0\leq i\leq M$.
Second we conclude for $0\leq i\leq N$,
\begin{align*}
 r(i)R(i,i+1)&=r(i)S(i,i+1)\min\left\{1,\frac{r(i+1)}{r(i)}\right\}\\
  &=r(i)\frac{1}{2} \min\left\{1,\frac{\choose{M}{i+1}}{\choose{M}{i}}\right\}\\
  &=\begin{cases}
      r(i)\frac{1}{2}\cdot\frac{M-i}{i+1}&\text{if }i\geq \frac{M-1}{2}\\
      r(i)\frac{1}{2}&\text{otherwise}
    \end{cases}\\
  &=\begin{cases}
      \frac{1}{2^{M+1}}\choose{M}{i}\cdot\frac{M-i}{i+1}&\text{if }i\geq \frac{M-1}{2}\\
      \frac{1}{2^{M+1}}\choose{M}{i}&\text{otherwise}
    \end{cases}
\end{align*}
Fixing $A:=4aMe^{\beta \frac{N(K+1)}{M}}$ it is now straightforward to check that
\begin{equation}
 r(i)R(i,i+1)\leq A\pi(\Omega_i)\overline Q(i,i+1)
\end{equation}
holds, for any $i$. Now Lemma \ref{Lemma_Compare_Derichlet_forms} in the appendix yields the desired inequality
\begin{equation}
 \frac{1}{4M^2}e^{-\beta \frac{N(K+1)}{M}}=\frac{a}{A}\cdot\frac{1}{M}\leq \frac{a}{A}\Gap(R)\leq \Gap(\overline Q).
\end{equation}
\end{proof}

\subsection{Speed of convergence of $Q P_+ Q$}\label{Speed2}

Ellis et al. \cite{EllisOttoTouchetteBEG} show a continuous phase transition in the state space for these values of $K$. All but exponential little mass is located around
\begin{equation}
 a_{\max}(0):=\left(\frac{e^{-\beta}}{1+2e^{-\beta}},\frac{1}{1+2e^{-\beta}},\frac{e^{-\beta}}{1+2e^{-\beta}}\right)\in \Upsilon_\infty
\end{equation}
for $\beta<\beta_c^{(2)}(K)$ and for $\beta>\beta_c^{(2)}(K)$ all but exponential little mass is located around the points
\begin{align}
 a_{\max}(-1)&:=\left(\frac{e^{2\beta K z_\alpha -\beta}}{C(\beta,K)},\frac{1}{C(\beta,K)},\frac{e^{-2\beta Kz_\alpha-\beta}}{C(\beta,K)}\right)\in \Upsilon_\infty\\
a_{\max}(1)&:=\left(\frac{e^{-2\beta Kz_\alpha-\beta}}{C(\beta,K)},\frac{1}{C(\beta,K)},\frac{e^{2\beta Kz_\alpha-\beta}}{C(\beta,K)}\right)\in \Upsilon_\infty
\end{align}
with $C(\beta,K)=1+e^{-2\beta Kz_\alpha-\beta}+e^{2\beta Kz_\alpha-\beta}$ being the normalization constant and $z_\alpha(\beta,K)\geq 0$ as constructed but not computed in \cite{EllisOttoTouchetteBEG}, also see the appendix for an insight in the technical problems one faces. The standard Metropolis chain would get stuck in either of the regions around $a_{\max}(1)$ or $a_{\max}(-1)$ as it is exponentially unlikely for the chain to leave either of these local states. The swapping chain circumvents this bottleneck by swapping a component located close to $a_{\max}(-1)$ up to $\beta<\beta_c^{(2)}(K)$ at which temperature the Metropolis chain is rapidly mixing on the whole state space. It will find a state close to $a_{\max}(0)$ and, if suggested to increase $\beta$, it will choose either of the two paths leading to $a_{\max}(-1)$ or $a_{\max}(1)$ with equal probability. The bottleneck encountered in the intermediate regime $\Klow<K<K_c$, which is described and used in Section \ref{Subsection-Case_Klow<K<Kc},
will not pose a problem, as
\begin{equation}
\beta \mapsto \begin{cases}
                a_{\max}(0)&\text{if } \beta\leq\beta_c^{(2)}(K)\\
		a_{\max}(1)&\text{if } \beta>\beta_c^{(2)}(K)
              \end{cases}
\end{equation}
is continuous in the present case $K>K_c$.

To formalize this, a technique introduced by Bhatnagar and Randall \cite[Sec. 4.1]{BhatnagarRandallTorpidMixingPotts} (in a modified form) will prove to be a powerful tool for showing rapid mixing of $Q P_+ Q$. We need to recall the notation of $\A_{s,r}$
introduced in \eqref{A_sr}. Assume $\beta$ is big enough, such that the function $f_\beta$ introduced in \eqref{EnergyLandscape} on the field $\A=(A_{s,r})_{s,r}$ has two local maxima, such that it has two local modes. Inspired by \eqref{EnergyLandscape_Inducing-P_f} we define a probability measure $P_{f_\beta}$ on $\mathcal{B}:=\{(a_{-1},a_1)\in[0,1]^2|a_{-1}+a_{1}\leq1 \text{ and } a_{-1}\leq a_1\}$ by
\begin{equation}
 \frac{dP_{f_\beta,N}}{d\lambda}(a_{-1},a_1):=\frac{1}{Z_{f_\beta}(N)} e^{N f_\beta(a_{-1},1-a_{-1}-a_1,a_1)} \label{Definition-P_f}
\end{equation}
where $\lambda$ denotes the Lebesgue-Measure restricted to the subset $\mathcal{B}$. $Z_{f_\beta}(N)$ denotes the normalization constant. Let $a_g(\beta_{i_c})$ denote the unique local maximum point of $f_{\beta_{i_c}}$ on $\mathcal{B}$ at the next to critical temperature
$$i_c:=\max\{i|\beta_i\leq \beta_c^{(2)}(K)\}.$$
Further define the set
$$\mathcal{V}:=\{a_{\max}(1)|\beta\geq\beta_c^{(2)}(K)\}$$
which defines a continuous path from $a_{\max}(0)(\beta_c^{(2)}(K))$ to $(0,0,1)$ in $\mathcal{B}$. Take $\mathcal{V}$ to be an ordered set with the previously implied ordering. The path $\mathcal{V}$ separates $\mathcal{B}$ into two disjoint parts $\mathcal{B}_g\cup\mathcal{B}_l=\mathcal{B}$ with $\mathcal{V}\subseteq \mathcal{B}_g$. Obviously
$$P_{\beta_c^{(2)}(K),N}(\mathcal{B}_g)=1-P_{\beta_c^{(2)}(K),N}(\mathcal{B}_l)\to c\in(0,1)$$
for some $K$-specific constant $c$ as $N\to \infty$.
Remembering the models phase behavior we will define $\mathcal{B}_g$ and $\mathcal{B}_l$ by $(\frac12,\frac12)\in\mathcal{B}_g$ while $(0,0)\in\mathcal{B}_l$ as this notation reflects where the global and local maxima appear.
With the definition of
$$\A_g(\beta_{i_c}):=\big(\mathcal{B}_g\cap\Upsilon\big) \text{ and } \A_l(\beta_{i_c}):=\big(\mathcal{B}_l\cap\Upsilon\big)$$
we know by continuity of $\pi_\beta$ in $\beta$ that $\pi_{\beta_{i_c}}(\A_g(\beta_{i_c}))\goesto c$ and consequentially $\pi_{\beta_{i_c}}(\A_l(\beta_{i_c}))\goesto 1-c$. For any $i\in\{i_c+1,...,M\}$ there exist two local maxima, the global one denoted by $a_g(\beta_i)$ and the local (non-global) one denoted by $a_l(\beta_i)$. We define $\A_g(\beta_i)$ and $\A_l(\beta_i)$ by
\begin{align}
\text{there is no nondecreasing path from $a$ to $a_l$}&\follows a\in\A_g(\beta_i)\label{Temperatur_dependant_partitioning_2}\\
\text{there is no nondecreasing path from $a$ to $a_g$}&\follows a\in\A_l(\beta_i)\\
\begin{array}{r}\text{there exist nondecreasing paths}\\\text{from $a$ to $a_g$ and from $a$ to $a_l$}\end{array}&\follows \begin{cases}
	a\in\A_g(\beta_i)&\text{if }a\in\A_g(\beta_{i-1})\\
	a\in\A_l(\beta_i)&\text{if }a\in\A_l(\beta_{i-1})\\
	\end{cases}\label{Temperatur_dependant_partitioning_3}
\end{align}
Note that for each $i$ the sets  $\A_g(\beta_i)$ and $\A_l(\beta_i)$ form a partition of $\mathcal B$, since otherwise $f_\beta$ would need to have more than two maxima on
$\mathcal B$, in contradiction to Theorem \ref{Theorem_Holger}.
It will prove convenient to have
\begin{lemma}\label{Lemma_monotonc-behaviour-of-A_g}
$(\pi_i(\A_g(\beta_i))_{i\in\{i_c,...,M\}}$ is monotonically increasing, while $(\pi_i(\A_l(\beta_i))_{i\in\{i_c,...,M\}}$ is monotonically decreasing.
\end{lemma}
\begin{proof}
This proof consists of multiple parts. We will first establish that for $\beta\geq \beta_c^{(2)}(K)$
\begin{align}
 f_\beta(a_{\max}(0))&\quad \text{is monotonically decreasing, while}\label{Equation_BEG_K>K_c_f(a_l)MonotonicallyDecreasing}\\
 f_\beta(a_{\max}(1))&\quad\text{is monotonically increasing.}\label{Equation_BEG_K>K_c_f(a_g)MonotonicallyIncreasing}
\end{align}
This is a straightforward calculation. Inserting $a_{\max}(0)(\beta)$ into $f_\beta$ yields
$$\frac{d f_\beta(a_{\max}(0))}{d \beta}=-\frac{2 e^{-\beta}}{1+2e^{-\beta}}<0$$
thus \eqref{Equation_BEG_K>K_c_f(a_l)MonotonicallyDecreasing}. Defining the canonical free energy of a thermodynamical system by
\begin{equation}
 \phi(\beta):=\lim_{N\to\infty} \frac{1}{N} \log\big(Z_\beta(N)\big)
\end{equation}
it follows from \eqref{EnergyLandscape_Inducing-P_f} that in the interesting phase of $\beta\geq \beta_c^{(2)}(K)$
\begin{equation}\phi(\beta)=f_\beta(a_{\max}(1)),\label{BEG-VaradhansLemmaSubstitute}\end{equation}
as
\begin{eqnarray*}
\phi(\beta)&= &\lim_{N\to\infty} \frac{1}{N} \log\big(Z_\beta(N)\big)\notag\\
  &=&\lim_{N\to\infty} \frac{1}{N} \log\big( \sum_{\textbf{a}\in\Upsilon_N} e^{N f_\beta(\textbf{a})} \big)\notag\\
  &\leq& \lim_{N\to\infty} \frac{1}{N} \log\big(N^2 e^{N f_\beta(a_{\max}(1))}\big)\notag\\
  &=& \lim_{N\to\infty} \big(\frac{2}{N} \log(N)+f_\beta(a_{\max}(1))\big)\notag\\
  &=&f_\beta(a_{\max}(1))\notag\\
  \end{eqnarray*}
  \begin{align}
\phi(\beta)&=\lim_{N\to\infty} \frac{1}{N} \log\big(Z_\beta(N)\big)\notag\\
  &=\lim_{N\to\infty} \frac{1}{N} \log\big( \sum_{\textbf{a}\in\Upsilon_N} e^{N f_\beta(\textbf{a})} \big)\notag\\
  &\geq \lim_{N\to\infty} \frac{1}{N} \log\big(e^{N f_\beta(a_{\max}(1))}\big)\notag\\
  &= f_\beta(a_{\max}(1)).\notag
\end{align}
Differentiating for a fixed state $x=(x_{-1},x_0,x_{+1})$ in the domain of $f_\beta$ gives us
\begin{equation}
 \frac{d f_\beta}{d\beta}(x)=x_0-1+K(x_{1}-x_{-1})^2 \label{Equation_BEG_DerivativeOf_f_a_g}
\end{equation}
which implies
$$\frac{d f_\beta}{d\beta}(x) \xrightarrow{x\to (0,0,1)}K-1>0.$$
This guarantees $f_\beta(a_{\max}(1))$ to be strictly increasing for sufficiently large $\beta$. Together with the general fact (see for instance \cite{EllisHavenTurkington-LargeDeviationPrinciplesAndCompleteEquivalence} or, for a non-rigorous overview, \cite{EllisTouchetteTurkington-ThermodynamicVersusSTatisticalNonequivalenceOfEnsemblesBEG}) that $\phi(\beta)$ is concave for $\beta>\beta_c^{(2)}(K)$ we gain \eqref{Equation_BEG_K>K_c_f(a_g)MonotonicallyIncreasing}.

In the second step we will confirm, that there is no point-movement from $\A_g$ to $\A_l$ by going from $\beta_i$ to $\beta_{i+1}$ for all $i_c\leq i\leq M-1$. For this, first note, that any point $x$, which has a nondecreasing path to any point $y\in\mathcal{V}$ also has a nondecreasing path to $a_g$. Assume, this to be wrong:

First note, that $f_0$ is monotonically decreasing on $\mathcal{V}$. Assume it would not be, then there are two points, $z_1,z_2\in\mathcal{V}$ with $f_0(z_1)=f_0(z_2)$. As $a_{\max}(1)$ is continuously moving from $a_{\max}(0)(\beta_c^{(2)}(K))$ to $(0,0,1)$ there needs to be a $\beta'>\beta_c^{(2)}(K)$ such that $f_{\beta'}(z_1)>f_{\beta'}(z_2)$. Of course, there also needs to be a $\beta''>\beta'$ such that $f_{\beta''}(z_1)<f_{\beta''}(z_2)$. This contradicts \eqref{Equation_BEG_DerivativeOf_f_a_g}.

Coming back to the original contradiction argument: By assumption, there exists a $\beta>\beta_c^{(2)}(K)$ such that $f_\beta$, if restricted to $\mathcal{V}$, has at least two modes -- where, without loss of generality, the highest one is in the one containing $a_{\max}(0)(\beta_c^{(2)}(K))$. Take $z\in\mathcal{V}$ to be a local minimum. The points $z'$ just further away from $a_{\max}(0)(\beta_c^{(2)}(K))$ than $z$ must thus satisfy
$$\frac{d f_\beta}{d\beta}(z)<\frac{d f_\beta}{d\beta}(z'),$$
as $f_0$ is monotonically decreasing on $\mathcal{V}$ and the derivative of $f_\beta$ with respect to $\beta$ does not depend on $\beta$.
This warrants for $f_\beta(z)<f_{\beta'}(z')$ for all $\beta'>\beta$ (again for the same reason), which in turn implies either $a_{\max}(1)$ stays left of $z$ for all $\beta$ or that $a_{\max}(1)$ exhibits a discontinuous behavior close to $z$. Both contradict a combination of Theorem \ref{Theorem_Holger} and the continuity of $a_{\max}(1)$.

This directly implies, that every point $x\in\A_g(\beta_{i_c})$ stays in $\A_g$ for all $i$, as any (nondecreasing) path leading from $x$ to $a_l(\beta_i)$ will need to cross the set $\mathcal{V}$. A point $x\in\A_g(\beta_i)$ which does not lie in $\A_g(\beta_{i_c})$ must have been forced to switch from $A_l$ to $A_g$ at some index $i_c<j\leq i$. This means $x$ is being separated from $a_l$ by some path. Due to an argument close to the one given before, this path will block the way from $x$ to $a_l$ for any $i\geq j$, such that again, $x\in\A_g(\beta_{i+1})$.

Now, for any $\beta>\beta_c^{(2)}(K)$ it follows from a similar calculations as for equation \eqref{BEG-VaradhansLemmaSubstitute}, that
\begin{align}
 \lim_{N\to\infty} \frac{1}{N} \log\big(\pi_{\beta_i}(\A_g)\big)&=0\\
 \lim_{N\to\infty} \frac{1}{N} \log\big(\pi_{\beta_i}(\A_l)\big)&=f_{\beta_i}(a_{\max}(0))-f_{\beta_i}(a_{\max}(1))
\end{align}
which together with the first and second argument yields the claim.
\end{proof}

For later use we need the following partitioning of the state space.
\begin{definition}[Definition 4.1 of \cite{BhatnagarRandallTorpidMixingPotts}]\label{Definition_Trace}
For $x\in\OmegaP^{M}$ define the trace $$\tr(x)=t\in\{0,1\}^M$$ with $t_i=0 \iff x_i\in\A_l$ and $t_i=1 \iff x_i\in\A_g$ to indicate which part of the state space which component is in.\end{definition}

The $2^{M-i_c+1}$ possible values of $\tr(x)$ characterize the partitioning
\begin{equation}
 \OmegaP^M=\bigcup_{t\in\{0,1\}^M}\OmegaP_t \label{Temperatur_dependant_partitioning_1}
\end{equation}
(with the canonical definition of $\OmegaP_t$) we will use. First using Lemma \ref{Multiple_usage_does_not_change_Gap} in the appendix for \eqref{Equation_QPQ_1}, Lemma \ref{MadrasZhengLemma7} for \eqref{Equation_QPQ_2} and afterwards Theorem \ref{CaraccioloPelissettoSokal} we obtain
\begin{align}
\Gap(Q P_+ Q)&\geq \frac{1}{3}\Gap(Q P_+ QQ P_+ QQ P_+ Q)\label{Equation_QPQ_1}\\
&\geq \frac{1}{3}\Gap((Q P_+ Q)^{\frac12}Q P_+ Q(Q P_+ Q)^{\frac12})\label{Equation_QPQ_2}\\
&\geq \frac{1}{3}\Gap(\widehat Q)\cdot \min_{t} \big\{\Gap\big((QP_+Q)|_{\tr^{-1}(t)}\big)\big\}
\end{align}
where $\hat Q$ is an abbreviation for the aggregated chain
$\overline{Q P_+ Q}$.
We can argue as in \eqref{Equation_QPQ-RestrictedTo_k} to get
\begin{align}
 \Gap(Q P_+ Q)&\geq\frac{1}{3}\Gap(\widehat Q)\cdot \min_{t} \big\{\Gap\big((QP_+Q)|_{\tr^{-1}(t)}\big)\big\}\notag\\
  &\geq \frac{1}{24}\Gap(\widehat Q)\cdot \min_{t} \big\{\Gap(Q|_{\tr^{-1}(t)}P_+|_{\tr^{-1}(t)}Q|_{\tr^{-1}(t)})\big\}\notag\\
  &\geq \frac{1}{24}\Gap(\widehat Q)\cdot \min_{t} \big\{\Gap(P_+|_{\tr^{-1}(t)})\big\}\label{Gap_inequality_1}
\end{align}
where the last inequality uses Lemma \ref{MadrasZhengLemma7} again.
This looks promising, as the set $\tr^{-1}(t)$
is unimodal in each component as constructed, and thus the chain $P_+|_{\tr^{-1}(t)}$ should be fast on this subset. $\widehat Q$ will be comparable to a very simple random walk, which is known to be rapidly mixing, thus leading to a polynomial lower bound for $\Gap(Q P_+ Q)$.

\subsubsection{Speed of convergence of the aggregated chain $\widehat Q$}\label{Aggregated_Chain_Qhat}
In the wake of Bhatnagar and Randall \cite[Theorem 4.4]{BhatnagarRandallTorpidMixingPotts} we define the probability measure
\begin{equation} \widehat \pi(t):=\prod_{i=1}^M \pi_i\Big(\tr^{-1}_i(t)\Big)\end{equation}
on the state space
\begin{equation} \OmegaHat=\prod_{i=1}^{i_c-1} \{1\}\times \prod_{i=i_c}^{M}\{0,1\}.\label{Temperatur_dependant_partitioning_4}\end{equation}
A simple reversible random walk $\widehat{RW1}$ with respect to $\widehat \pi$ to compare $\widehat Q$ on $\OmegaHat$ to would be the following. Start at some $t\in\OmegaHat$ and either switch the component $t_{i_c}$ from $0$ to $1$ or vice versa with the Metropolis probabilities induced by $\widehat \pi$, or choose an $i\in\{i_c,...,M-1\}$ at random and interchange components $i$ and $i+1$ according to a Metropolis update with regard to $\widehat \pi$ as well, such that $t\to (i,i+1)t$. Again, for technical reasons $\widehat{RW1}$ does not act on $t$ at all with probability $\frac12$. In order to analyze $\widehat{RW1}$ we will compare it with an even simpler random walk $\widehat{RW2}$  on $\OmegaHat$ which picks an $i\in\{i_c,...,M\}$ at random and updates $t_i$ by choosing $t'_i$ exactly according to the stationary distribution $\widehat\pi_i$. It is apparent, that after this move, the $i$th component of $t$ is in equilibrium. Using the coupon collector's theorem (see for
  instance  (2.7), (5.10) and (12.12) in \cite{LevinPeresWilmer}), we get easily
\begin{lemma}
\label{Rapid_mixing_of_RW2}
 Let $\widehat R$ denote the transition kernel of $\widehat{RW2}$ . Then
$$\Gap(\widehat R) \geq \frac 1{4 M\log M}.$$
\end{lemma}
This leads directly to
\begin{theorem}\label{Rapid_mixing_of_RW1}
 The aggregated chain $\widehat Q$ of the Swapping Markov chain is rapidly mixing on $\OmegaHat$ for $K> K_c$.
\end{theorem}
\begin{remark}
Again we refrain from giving an explicit bound (also recall the remark after Theorem \ref{Klarge}).
\end{remark}
\begin{proof}
\newcommand{\Rhat}{\widehat{R}}

 The main idea is, to give a canonical path in $\widehat {RW1}$ in which every step compares well to the rapidly mixing chain $\widehat R$. Consider a single transition $(t,t')$ in $\Rhat$, thus $t'=(t_1,...,t_{i-1},1-t_i,t_{i+1},...,t_M)$ for one $i\geq i_c$. Now consider the concatenation $p_1\circ p_2\circ p_3$ of the three paths
\begin{itemize}
\item $p_1$ consists of the $i-i_c$ swap moves from $t$ to $$t^{(1)}=(t_1,...,t_{i_c-1},t_i,t_{i_c},...,t_{i-1},t_{i+1},...,t_M)$$
\item $p_2$ is the one step from $t^{(1)}$ to $$t^{(2)}=(t_1,...,t_{i_c-1},1-t_i,t_{i_c},...,t_M)$$
\item $p_3$ consists of the $i-i_c$ steps needed to swap the $i$th component back up, thus $p_2$ is the path from $t^{(2)}$ to $$t^{(3)}=(t_1,...,t_{i_c},...,t_{i-1},1-t_i,...,t_M).$$
\end{itemize}
In order to be able to use Lemma \ref{DSCComparison} in the appendix we will establish that
\begin{equation}
 \widehat \pi(z) \widehat{RW1}(z,z')\geq \frac 12\ \widehat \pi(t)\Rhat(t,t')\label{RandallEq4.1}
\end{equation}
holds for any transition $(z,z')$ in the canonical path $p_1\circ p_2\circ p_3$.

\underline{Transition along $p_1$}: Let $z=(t_0,...,t_{i_c},...,t_{j-1},t_i,t_{j},...,t_M)$ for a $j\in\{i_c+1,...,M\}$ and $z'=(j-1,j)z$. It is easy to verify
\begin{align}
 \widehat\pi(z)\widehat{RW1}(z,z')&=\frac{\piHat(z)}{2(M-i_c+1)} \min\left(1,\frac{\piHat(z')}{\piHat(z)}\right)\notag\\
&=\frac{1}{2(M-i_c+1)}\min(\piHat(z),\piHat(z'))
\end{align}
and for $t, t'=(t_1,...,t_{i-1},1-t_i,t_{i+1},...,t_M)$ for one $i\geq i_c$,
\begin{align}
 \piHat(t)\Rhat(t,t')&=\frac{\piHat(t)}{(M-i_c+1)}   \hat\pi_i(t'_i) \notag\\
&\leq\frac{1}{(M-i_c+1)}\piHat(t^*)
\end{align}
with
$$t^*=(t_1,...,t_{i-1},0,t_{i+1},...,t_M).$$
Thus it suffices to show $\piHat(t^*)\leq \piHat(z)$ and $\piHat(t^*)\leq \piHat(z')$. We will show this for $z$ only, as the argument works exactly the same for both $z$ and $z'$. It is useful to partition $t^*$ into blocks of bits $t_l$ that equal 1, separated by one or more zeros. Let $i_c\leq k<i$ be the largest value that satisfies $t_k=0$. Using Lemma \ref{Lemma_monotonc-behaviour-of-A_g}, it is straightforward to verify
$$\prod_{l=k+1}^{i} \piHat_l(z_l)\geq\prod_{l=k+1}^{i}\piHat_l(t_l^*).$$
Similarly , consider the next block of $1$s in $t^*$, until the first index $k'$ such that $t_k'=0$,
$$\prod_{l=k'+1}^{k}\piHat_l(z_l)\geq\prod_{l=k'+1}^{k}\piHat_l(t_l^*).$$
Continuing in this way we find
$$\prod_{l=j}^{i}\piHat_l(z_l)\geq\prod_{l=j}^{i}\piHat_l(t_l^*)$$
and thus
$$\piHat(z)\geq\piHat(t^*).$$
In an analogous fashion one can also show
$$\piHat(z')\geq\piHat(t^*)$$
such that \eqref{RandallEq4.1} holds on all transitions in $p_1$.

\medskip
\underline{Transition along $p_2$}: The same argument as before yields $$\min(\piHat(z),\piHat(z'))\geq\piHat(t^*)$$ for $(z,z')\in p_2$.

\medskip
\underline{Transition along $p_3$}: This is exactly as the case of $p_1$.\\[2ex]
We find, that for any edge $(z,z')$ in the canonical path equation \eqref{RandallEq4.1} is satisfied, so what needs to be done in order to show rapid convergence of $\widehat{RW1}$ to equilibrium is to ensure that not too many paths use the same transition $(z,z')$. With the notation of Lemma \ref{DSCComparison} below, we can obviously bound  the number of paths in $\tilde E(z,z')$ by $M$  and as any path $\gamma_{t,t'}$ has at most $2M+1$ many transitions, we can guarantee
\begin{equation}
 A=\max_{(z,z')}\left\{\frac{1}{\piHat(z)\widehat{RW1}(z,z')}\sum_{\tilde E(z,z')}|\gamma_{t,t'}|\piHat(t)\Rhat(t,t')\right\}\leq 4M^2+2M
\end{equation}
which leads to $\Gap(\widehat{RW1})\geq \big(2(2M^3+M^2)\log(M)\big)^{-1}$.

It remains to compare $\widehat{RW1}$ with $\Qhat$. We will do so by means of case differentiation. First consider the case of $z'=(i,i+1)z$ with $z_i=1$, $z_{i+1}=0$ in which we will show
\begin{equation}
 \Qhat(z,z')\geq \frac{1}{8M}e^{-\beta\frac{N(K+1)}{M}} \widehat{RW1}(z,z'),
\end{equation}
where the term $e^{-\beta\frac{N(K+1)}{M}}$ is of order $O(1)$ as $M= c_1 N$.
So taking $z'=(i,i+1)z$ with $z_i=1$, $z_{i+1}=0$ leads to
\begin{equation}
 \widehat{RW1}(z,z')=\frac{1}{2(M-i_c+1)} \min\left(1,\frac{\piHat(z')}{\piHat(z)}\right)=\frac{1}{2(M-i_c+1)}\label{RW1Calc1}
\end{equation}
as $\piHat_i(1)\leq \piHat_{i+1}(1)$ and $\piHat_i(0)\geq \piHat_{i+1}(0)$. The equivalent for $\Qhat$ yields with $\mathcal{B}:=\big\{x\in\Omega_{+z}\big|x_i\in B_\epsilon(a_g)\cap\A_g, x_{i+1} \in B_\epsilon(a_l)\cap\A_l\big\}$
\begin{eqnarray}
&&\frac{1}{\piHat(z)}\sum_{x\in \Omega_{+z}}\sum_{y\in \Omega_{+z'}}\pi(x) \big(QP_+Q\big)(x,y)\notag\\
 &\geq& \frac{1}{4\piHat(z)}\sum_{x\in \Omega_{+z}}\sum_{y\in \Omega_{+z'}}\pi(x) Q(x,y)\notag\\
	&=&\frac{1}{4\piHat(z)}\sum_{x\in \Omega_{+z}} \pi(x) Q(x,(i,i+1)x)\notag\\
	&=&\frac{1}{4\piHat(z)}\left(\sum_{x\in\mathcal{B}} \pi(x) Q(x,(i,i+1)x) + \sum_{x\in\Omega_{+z}\setminus\mathcal{B}} \pi(x) Q(x,(i,i+1)x)\right)\notag\\
	&\geq& \frac{1}{4\piHat(z)}\sum_{x\in\mathcal{B}} \pi(x) Q(x,(i,i+1)x)\notag\\
	&\ge&\frac{1}{4\piHat(z)} \frac{1}{2(M+1)}e^{-\beta\frac{N(K+1)}{M}}\pi(\mathcal{B}) \label{Calc1Arg1}\\
	&\geq& \frac{1}{8(M+1)}e^{-\beta\frac{N(K+1)}{M}}(1-e^{-cN})\label{Calc1Arg2}.
\end{eqnarray}
Equation \eqref{Calc1Arg1} is obtained analogously to \eqref{Lemma61_1}. For \eqref{Calc1Arg2} we use Theorem \ref{Theorem_Holger}, which implies that
$$\frac{\pi_i(B_\epsilon(a_g)\cap\A_g)}{\pi_i(A_g)}\frac{\pi_{i+1}(B_\epsilon(a_l)\cap\A_l)}{\pi_{i+1}(A_l)}\geq 1 - e^{-c N},$$ for some  $c>0$. Second consider $z'=(i,i+1)z$ with $z_i=0$, $z_{i+1}=1$ which leads to
\begin{equation}
 \widehat{RW1}(z,z')=\frac{1}{2(M-i_c+1)} \min\left(1,\frac{\piHat(z')}{\piHat(z)}\right)=\frac{1}{2(M-i_c+1)}\frac{\piHat(z')}{\piHat(z)}\label{RW1Calc2}
\end{equation}
and with $\mathcal{B}':=\big\{x\in\Omega_{+z}\big|x_i\in B_\epsilon(a_l)\cap\A_l, x_{i+1} \in B_\epsilon(a_g)\cap\A_g\big\}$
\begin{align}
\frac{1}{\piHat(z)}&\sum_{x\in \Omega_{+z}}\sum_{y\in \Omega_{+z'}}\pi(x) \big(QP_+Q\big)(x,y)\notag\\
	&\geq \frac{1}{4\piHat(z)}\sum_{x\in \Omega_{+z}}\sum_{y\in \Omega_{+z'}}\pi(x) Q(x,y)\notag\\
	&=\frac{1}{4\piHat(z)}\sum_{x\in \Omega_{+z}} \pi(x) Q(x,(i,i+1)x)\notag\\
	&\geq \frac{1}{4\piHat(z)}\sum_{x\in\mathcal{B}'} \pi(x) Q(x,(i,i+1)x)\notag\\
	&\ge\frac{1}{4\piHat(z)}\sum_{x\in\mathcal{B}'} \frac{1}{2(M+1)}e^{-\beta\frac{N(K+1)}{M}} \pi\big((i,i+1)x\big)\label{Calc2Arg1}\\
	&=\frac{1}{4\piHat(z)} \frac{1}{2(M+1)}e^{-\beta\frac{N(K+1)}{M}}\piHat(z') \notag\\
	&\geq \frac{1}{8(M+1)} \frac{\piHat(z')}{\piHat(z)}e^{-\beta\frac{N(K+1)}{M}}(1-e^{-cN})\label{Calc2Arg2}.
\end{align}
The arguments for \eqref{Calc2Arg1} and \eqref{Calc2Arg2} are the same as above. The two remaining cases of $z'=(z_0,...,1-z_{i_c},...,z_M)$ with $z_{i_c}\in\{0,1\}$ are dealt with automatically  by showing rapid mixing of $P_{i_c}$ on $\A_g=\A$. The claim follows by using Lemma \ref{Lemma_Compare_Derichlet_forms}.
\end{proof}

\subsubsection{Rapid Mixing in $\A_g$ and $\A_l$}\label{BEGequienergy}
It remains to show rapid convergence to equilibrium of $P_+|_{\tr^{-1}(t)}$ as constructed in \eqref{Gap_inequality_1}.
Using Theorem \ref{DiaconisSaloffCosteProductChain} we can stick to the case of
$$T:=P_i|_{\tr_i^{-1}(t)}$$
for fixed $t$ and $i$.
Using Lemma \ref{Multiple_usage_does_not_change_Gap} with $m=3$ gives us
$$\Gap(T)\geq\frac13\Gap(T^3)$$
which will prove to be simpler to handle than $T$ itself. We will only deal with the case of $\A_g$ as the case of $\A_l$ works the same. Consider the disjoint union
\begin{equation}
 \A_g=\bigcup_{\A_{s,r}\subseteq\A_g} \A_{s,r}
\end{equation}
and decompose the state space accordingly. This leads to
\begin{equation}
 \Gap(T^3)=\Gap(T^{\frac12}T^2 T^{\frac12})\geq \Gap(\overline{T})\cdot\min_{s,r}\Gap(T^2_{s,r})
\end{equation}
which may now make apparent, why dealing with $T^3$ is an advantage over dealing with $T$.
Here $\overline T$ is the aggregated chain defined as $\overline{\mathcal Q}$ in Theorem \ref{CaraccioloPelissettoSokal}.
Restricting $T^2$ to $\A_{s,r}$ will still give us a nontrivial chain, whilst the restriction of $T$ to $\A_{s,r}$ would deterministically stay in the originally occupied state.
\begin{theorem}
$\Gap(\overline{T})\geq \frac{1}{4}N^{-5}$
\end{theorem}
\begin{proof}
 This is already well prepared. As constructed earlier, 
 $f_\beta$ fulfills an unimodality condition on $\A_g$. Thus we can easily choose one path $\gamma_{xy}$ for any given set $x$ and $y$ that is unimodal. Each such path has at most length $N^2$, such that the Poincar\'e inequality given in Lemma \ref{Lemma_Poincare_Inequality} simplifies to
\begin{align}\label{Calculation_K}
A&=\max_{\langle \A_{s,r},\A_{s',r'}\rangle} \frac{1}{\pi_i(\A_{s,r})\overline{T}(\A_{s,r},\A_{s',r'})}\sum_{\gamma_{z_1 z_2}\ni \langle \A_{s,r},\A_{s',r'}\rangle} |\gamma_{z_1 z_2}| \pi_i(z_1)\pi_i(z_2)\notag\\
 &\leq N^2 \max_{\langle \A_{s,r},\A_{s',r'}\rangle} \frac{1}{\pi_i(\A_{s,r})\overline{T}(\A_{s,r},\A_{s',r'})}\sum_{\gamma_{z_1 z_2}\ni \langle \A_{s,r},\A_{s',r'}\rangle} \pi_i(z_1)\pi_i(z_2)\notag\\
 &= N^2 \max_{\langle \A_{s,r},\A_{s',r'}\rangle} \sum_{\gamma_{z_1 z_2}\ni \langle \A_{s,r},\A_{s',r'}\rangle} \frac{\pi_i(z_1)\pi_i(z_2)}{\pi_i(\A_{s,r})\overline{T}(\A_{s,r},\A_{s',r'})} \notag\\
\end{align}
It is now of interest, how $\overline T$ behaves. Given $\A_{s,r}\neq \A_{s',r'}$ with $\overline T(A_{s,r},\A_{s',r'})>0$, we first consider the case $\pi_i(\sigma)\leq \pi_i(\sigma')$ for $\sigma\in\A_{s,r}$ and $\sigma'\in\A_{s',r'}$. Note that $\pi_i(\sigma)$ is independent of the choice of $\sigma \in\A_{s,r}$.
\begin{align}
 \overline T(\A_{s,r},\A_{s',r'})&=\frac{1}{\pi_i(\A_{s,r})} \sum_{\sigma\in\A_{s,r}}\sum_{\sigma'\in\A_{s',r'}} \pi_i(\sigma) T(\sigma,\sigma')\notag\\
 &= \frac{1}{\pi_i(\A_{s,r})} \sum_{\sigma\in\A_{s,r}}\sum_{\sigma'\in\A_{s',r'}} \pi_i(\sigma) \frac{1}{4N}\notag\\
&\geq \frac{1}{4N} \frac{1}{\pi_i(\A_{s,r})} \sum_{\sigma\in\A_{s,r}} \pi_i(\sigma)\notag\\
&= \frac{1}{4N}\notag
\end{align}
The second case $\pi_i(\sigma)>\pi_i(\sigma')$ uses The reversibility of $T$ together with
\begin{align}
 \overline T(\A_{s,r},\A_{s',r'})&=\frac{1}{\pi_i(\A_{s,r})} \sum_{\sigma\in\A_{s,r}}\sum_{\sigma'\in\A_{s',r'}} \pi_i(\sigma) T(\sigma,\sigma')\notag\\
 &=\frac{1}{\pi_i(\A_{s,r})} \sum_{\sigma\in\A_{s,r}}\sum_{\sigma'\in\A_{s',r'}} \pi_i(\sigma') T(\sigma',\sigma)\notag\\
 &=\frac{1}{4N} \frac{1}{\pi_i(\A_{s,r})} \sum_{\sigma'\in\A_{s',r'}}\sum_{\sigma\in\A_{s,r}} \pi_i(\sigma')\notag\\
 &\geq \frac{1}{4N} \frac{\pi_i(\A_{s',r'})}{\pi_i(\A_{s,r})}.\notag
\end{align}
To further analyze \eqref{Calculation_K} we will take the worst case scenario $\frac{\pi_i(\A_{s',r'})}{\pi_i(\A_{s,r})}<1$ and for inequality \eqref{Calculation_K_2} recall that all paths are unimodal:
\begin{align}
 A&\leq N^2 \max_{\langle \A_{s,r},\A_{s',r'}\rangle} \sum_{\gamma_{z_1 z_2}\ni \langle \A_{s,r},\A_{s',r'}\rangle} \frac{\pi_i(z_1)\pi_i(z_2)}{\pi_i(\A_{s,r})\overline{T}(\A_{s,r},\A_{s',r'})} \notag\\
 &\leq  4N^3 \max_{\langle \A_{s,r},\A_{s',r'}\rangle} \sum_{\gamma_{z_1 z_2}\ni \langle \A_{s,r},\A_{s',r'}\rangle} \frac{\pi_i(z_1)\pi_i(z_2)}{\pi_i(\A_{s,r})} \frac{\pi_i(\A_{s,r})}{\pi_i(\A_{s',r'})} \notag\\
 &=4N^3 \max_{\langle \A_{s,r},\A_{s',r'}\rangle} \sum_{\gamma_{z_1 z_2}\ni \langle \A_{s,r},\A_{s',r'}\rangle} \frac{\pi_i(z_1)}{\pi_i(\A_{s,r})} \frac{\pi_i(z_2)}{\pi_i(\A_{s',r'})} \pi_i(\A_{s,r}) \notag\\
 &\leq 4N^5. \label{Calculation_K_2}
\end{align}
\end{proof}
\begin{theorem}\label{Theorem_T^2_fast_on_A_sr}
 $\Gap(T^2_{s,r})\geq \frac{1}{96N^6}e^{-\beta-4K\beta}$.
\end{theorem}
\begin{proof}
  We need to consider two cases. The first is $\A_{N,N}$ in which case $|\A_{N,N}|=1$, such that $T^2_{N,N}$ is the constant chain, and therefore rapidly mixing. The other case is $\A_{s,r}$ with $s\leq \min\{r,N-1\}$. Let $\sigma,\sigma'\in\A_{s,r}$ with $\sigma\neq \sigma'$. We will compare $T^2_{s,r}$ with the Markov chain $(X_i)_i$ given in Appendix \ref{rapid_mixing_on_three_colorings}. Assume $(j,k)\sigma=\sigma'$ for some $j,k\in\{1,...,N\}$. Otherwise $T^2_{s,r}(\sigma,\sigma')=\P\big(X_{i+1}=\sigma'\big|X_i=\sigma\big)=0$. We know
$$\P\big(X_{i+1}=\sigma'\big|X_i=\sigma\big)=\frac{1}{N^2}$$
and
$$T^2_{s,r}(\sigma,\sigma')\geq T(\sigma,\tau)T(\tau,\sigma')$$
for a fixed $\tau$. It is obvious that either $T(\sigma,\tau)=\frac{1}{4N}$ or $T(\tau,\sigma')=\frac{1}{4N}$. Due to the symmetry assume
$$\tau:=(\sigma_1,...,\sigma_{j-1},\sigma_k,\sigma_{j+1},...,\sigma_k,...\sigma_N)$$
and conclude
\begin{align}
 T(\sigma,\tau)&=\frac{1}{4N}\min\Big\{ 1,\frac{e^{\beta(N-R(\tau))-\frac{\beta K}{N} S^2(\tau)}} {e^{\beta(N-r)-\frac{\beta K}{N}s^2}} \Big\}\notag\\
&=\frac{1}{4N}\min\Big\{ 1,e^{\beta(r-R(\tau))+\frac{\beta K}{N}\big(s^2- S^2(\tau)\big)} \Big\}\notag\\
&=\frac{1}{4N}\min\Big\{ 1,e^{\beta(r-R(\tau))+\frac{\beta K}{N}\big(s-S(\tau)\big)\big(s+ S(\tau)\big)} \Big\}\notag\\
&\geq \frac{1}{4N} e^{-\beta-4K\beta}\notag
\end{align}
such that taking $\tau=(\sigma_1,...,\sigma_{j-1},\sigma_k,\sigma_{j+1},...\sigma_N)$, where,  without loss of generality, $\sigma_k>\sigma_j$ yields
$$T^2_{s,r}(\sigma,\sigma')\geq \frac{1}{16N^2}e^{-\beta-4K\beta}.$$

And we can easily deduce from Lemma \ref{Coupling_Lemma} that
$\Gap(X)\geq \frac 1{6 N^4}$ (see \cite{LevinPeresWilmer} for instance).
Then Lemma \ref{Lemma_Compare_Derichlet_forms}
  proves the claim.
\end{proof}

\section{Proof of Theorem \ref{Kintermediate} }\label{Subsection-Case_Klow<K<Kc}
In this section we will prove Theorem \ref{Kintermediate}, which concerns the case $\Klow<K<K_c$. This is done in three parts. We first give the general idea, why slow mixing should be expected. We then support this idea with the necessary calculations in the remaining parts.

\subsection{The idea}
We will follow Gore and Jerrum \cite{GoreJerrumSW} in order to find a bad cut in the state space of BEG for $\beta>\beta_c^{(1)}(K)$. Using their technique we can show, that the Metropolis chain has to overcome an exponential barrier to leave any local maximum. We will show, that an $\epsilon$-stripe around the $0$-axis contains such a maximum, with $\epsilon$ independent of $\beta_i$. Intuitively speaking this leads to the following behavior of the Tempering chain. At $\beta_i$ close to $0$ the chain will find the unique global maximum on the $0$-axis. As of now the tempering chain is trapped in this $\epsilon$-stripe, as Ellis et al. \cite{EllisOttoTouchetteBEG} show a discontinuous behavior of the global maximum as $\beta_i$ passes through $\beta_c^{(1)}(K)$. Thus the chain will never get the chance to leave this $\epsilon$ stripe within polynomial time at any temperature, even though, at low temperature, this stripe has exponentially little mass.
\subsection{One bad cut for BEG's Metropolis chain}
Following the idea stated earlier, we show the existence of a bad cut within close proximity to the $0$-axis in the two-phase region. It is well known, due to Ellis et al. \cite{EllisOttoTouchetteBEG}, that
\begin{equation}
 a_{\max}(0):=\left(\frac{e^{-\beta}}{1+2e^{-\beta}},\frac{1}{1+2e^{-\beta}},\frac{e^{-\beta}}{1+2e^{-\beta}}\right)\in \Upsilon_\infty \label{a_max}
\end{equation}
is the unique global maximum for $\beta<\beta_c^{(1)}(K)$ and a local, non-global, maximum for $\beta>\beta_c^{(1)}(K)$.
Here
\begin{equation}\label{upsiloninfty}
 \Upsilon_\infty :=\{(a_{-1}, a_0, a_1) \in \R_+^3: \sum_i a_i=1\}
\end{equation}
is the set of all probability measures on three points.
They further show, that the phase transition for fixed $K$ at $\beta_c^{(1)}(K)$ is discontinuous, thereby granting us, uniformly in $\beta$, the existence of an $\epsilon>0$ such that
\begin{equation}
 \mathcal{N}:=\big\{\sigma\big||S_N(\sigma)|\leq N\cdot\epsilon\big\}\label{N_central}
\end{equation}
contains only this local maximum, and $f_\beta$ restricted to $B_\epsilon(a_{\max}(0))$ is unimodal for all $\beta>0$. It is even possible to show $f_\beta$ restricted to $\mathcal{N}$ to be unimodal for all $\beta$, see Lemma \ref{N_is_unimodal} for details.

Recalling Section \ref{section_Results} we have
\begin{align}
 \pi_\beta(\sigma \text{ has type } N\cdot\textbf{a})&=\frac{1}{ZN} e^{N\Big(\beta\big(K(a_{-1}-a_{1})^2-a_1-a_{-1}\big)-\sum_{i=-1}^{1}a_i\log a_i\Big)+\Delta(\textbf{a})}\notag\\
&=\frac{1}{ZN} e^{N f_\beta(\textbf{a})+\Delta(\textbf{a})}.\label{function_f}
\end{align}
which implies, that every local maximum of $f_\beta$ yields a locally exponential structure in $\pi$. This leads to exponentially low conductance $\Phi_\mathcal{N}$ for all $\beta>\beta_c^{(1)}(K)$, thereby implying slow mixing of the Metropolis Algorithm in this regime.

\subsection{The bad cut for BEG's Simulated Tempering chain}
Having low conductance $\Phi_\mathcal{N}$ for any $\beta>\beta_c^{(1)}(K)$ using the Metropolis Algorithm it is easy to generalize this to the Simulated Tempering chain. To this end define
\begin{equation}
 \mathcal{N}_{\text{edge}}:=\big\{\sigma\big|N\epsilon-1\leq|S_N(\sigma)|\leq N\cdot\epsilon\big\}\label{N_edge}
\end{equation}
and get
\begin{theorem}\label{Metropolis_Low_Conductance}
 Let $\mathcal{N}$ and $\mathcal{N}_{\text{edge}}$ be defined as in \eqref{N_central} and \eqref{N_edge}. For $\Klow< K< K_c$ and any $\beta\geq 0$,
there exists an $\epsilon>0$ such that for sufficiently large $N$,
\begin{equation}
 \frac{\pi_\beta(\mathcal{N}_{\text{edge}})}{\pi_\beta(\mathcal{N})}\leq e^{-c N}
\end{equation}
holds, with $c>0$ only depending on $K$.
\end{theorem}
\begin{proof}
Recall equation \eqref{function_f}
$$ \pi_\beta(\sigma \text{ has type } N\cdot\textbf{a})=\frac{1}{ZN} e^{N f_\beta(\textbf{a})+\Delta(\textbf{a})}$$
and verify that there are only polynomially (in $N$) many $\textbf{a}\in\Upsilon$ which satisfy $N\cdot \textbf{a}\in\mathcal{N}_{\text{edge}}$. Then, considering
$$f_\beta(\textbf{a})=\beta\big(K(a_{-1}-a_{1})^2-a_1-a_{-1}\big)-\sum_{i=-1}^{1}a_i\log a_i$$
and the results presented by Ellis et al. \cite{EllisOttoTouchetteBEG} it is clear, that $f_\beta$ has a local maximum at $a_{\max}(0)$ (see equation \eqref{a_max}). Due to $f_\beta$ being smooth in $a_{\max}$ it is clearly possible to find an $\epsilon>0$ such that $f_\beta$ is unimodal on $B_\epsilon(a_{\max})$. Due to the discontinuous behavior of the system at $\beta_c^{(1)}(K)$ for $K\in(\Klow,K_c)$ and as $f_\beta(\textbf{a})$ is smooth in all variables, including $\beta$, this $\epsilon$ can be chosen uniform in $\beta$.

Combining this with the exponential structure of \eqref{function_f} leads to the desired result
$$\frac{\pi_\beta(\mathcal{N}_{\text{edge}})}{\pi_\beta(\mathcal{N})}\leq e^{-c N}$$
with $c$ depending only on $K$ and sufficiently large $N$.
\end{proof}
This is the main ingredient for this section's main
\begin{theorem}
 Define $\mathcal{N}$ and $\mathcal{N}_{\text{edge}}$ as in Theorem \ref{Metropolis_Low_Conductance}. For $\Klow< K< K_c$ and  $\beta>\beta_c^{(1)}(K)$, let $\beta_i = \frac iM \beta$ for $i=0,\ldots, M$. Then for the Simulated Tempering Markov chain, the set
$$\mathcal{S}:=\{(x,i) | \ x\in\mathcal{N}, i=0,\ldots,M\}$$
satisfies $\Phi_{\mathcal{S}}\leq e^{-c N}$ with $c>0$.
\end{theorem}
\begin{remark}
For the definition of the conductance $\Phi_{\mathcal{S}}$ of a set $\mathcal{S}$, see Theorem \ref{Conductance_Lemma}.
\end{remark}

\begin{proof}
 \renewcommand{\S}{\mathcal{S}}
 \newcommand{\Nedge}{\mathcal{N}_{\text{edge}}}
 Using Theorem \ref{Metropolis_Low_Conductance} we get
\begin{align}
 \Phi_{\S}&=\frac{\sum_{x\in\S, y\notin\S} \pi(x) QPQ(x,y)}{\pi(\S)}\notag\\
&=\frac{\sum_{\beta_i}\sum_{x\in\Nedge} \pi_i(x) \sum_{x'\in\mathcal{N}^c} QPQ(x,x')}{\sum_{\beta_i}\sum_{x\in\mathcal{N}} \pi_i(x)}\notag\\
&\leq \frac{\sum_{\beta_i}\sum_{x\in\Nedge} \pi_i(x)}{\sum_{\beta_i}\sum_{x\in\mathcal{N}} \pi_i(x)}\notag\\
&= \frac{\sum_{\beta_i} \pi_i(\Nedge)}{\sum_{\beta_i} \pi_i(\Nedge) \frac{\pi_i(\mathcal{N})}{\pi_i(\Nedge)} }\notag\\
&\leq \frac{\sum_{\beta_i} \pi_i(\Nedge)}{e^{c N}\sum_{\beta_i} \pi_i(\Nedge)} \notag\\
&=e^{-c N}\notag
\end{align}

\end{proof}
This concludes the proof of Theorem \ref{Kintermediate} by using a variant of Theorem \ref{Conductance_Lemma} in the appendix: Indeed, we do not have $\pi(\S)\le 1/2$ for all $\beta>\beta_c^{(1)}(K)$, but as an easy extension of Theorem \ref{Conductance_Lemma} one obtains
$$ \Gap(Q P_{st} Q) \le \frac 1 {1-q} \Phi,$$
if we define
$$\Phi  = \min_{\mathcal{S} : \pi(\mathcal{S})\le q} \Phi_\mathcal{S},$$
for some $q\in (0,1)$.

As we chose $\beta > \beta_c^{(1)}(K)$, and set $\beta_i = \frac iM \beta$, for
$i=0, \ldots, M,$ there exists a  $p\in (0,1)$ such that  $\beta_i \le \beta_c^{(1)}(K)$, for $i\le pM$ and $\beta_i > \beta_c^{(1)}(K)$ for $i> pM$. For $\beta_i > \beta_c^{(1)}(K)$, we have $\pi_\beta( \mathcal{N}) \le 1/2$, since $a_{\max}(0)$ is a local maximum, which implies

$$\pi(\S) = \frac1{M+1} \sum_{i=0}^M \pi_{i}(\mathcal{N})    \leq q:=p + \frac{1-p}2  <1.$$

\appendix

\section{General preparations}
In this section we give some fundamental definitions and some well known lemmas on Markov chains from other articles. We state them in this section for the reader's convenience.

\begin{definition}
Let $\A$ be a sigma-field on a set $\Omega$. The total variation distance between two probability measures $\pi$ and $\tau$ on $(\Omega,\A)$ is defined by
$$d(\pi,\tau)_{\text{TV}}:=\sup\big\{|\pi(A)-\tau(A)|\big|A\in\A\big\}.$$
\end{definition}

The fundamental result for all that follows is
\begin{theorem}[Ergodic Theorem for Markov chains]\label{Theorem-ErgodicTheorem}
Let $(X_0,X_1,$ $X_2,...)$ be an irreducible aperiodic Markov chain with state space $\mathcal{S}=\{s_1,...,s_k\}$, transition matrix $P$ and arbitrary initial distribution $\mu^{(0)}$. Then there exists a unique distribution $\pi$ which is stationary for the transition matrix $P$. If $\mu^{(n)}$ denotes the distribution of $X_n$ then
\begin{equation*}
 \mu^{(n)} \xrightarrow{TV} \pi.
\end{equation*}
\end{theorem}

In general, the definition of stationarity proves complicated to construct or to verify for a given transition matrix $P$ or for a given probability distribution $\pi$. There is the tighter concept of reversibility which, in most cases, is much easier to construct.
\begin{definition}
Let $(X_0,X_1,...)$ be a Markov chain with state space $\mathcal{S}=\{s_1,...,s_k\}$ and transition matrix $P$. A probability distribution $\pi$ on $\mathcal{S}$ is said to be reversible for the chain if for all $x,y\in \mathcal{S}$ we have
$$\pi(x) P(x,y)=\pi(y) P(y,x).$$
The Markov chain is said to be reversible if there exists a reversible distribution for it.
\end{definition}

The key question for all kind of MCMC algorithms is how fast they mix, i.e. how rapidly they converge to the desired invariant measure. So in general, let $(X_n)_{n \ge 0}$ be a homogeneous, irreducible and aperiodic Markov chain on a finite state space $\Omega$, reversible with respect to a probability measure $\pi$ (on $\Omega$, that necessarily charges every point). The speed of convergence is determined in terms of
$$
\tau(\epsilon) = \min\{n: \mathrm{d_{TV}}( \mu^{(n)}, \pi) \le \epsilon\}.
$$
Here, of course, $ \mu^{(n)}$ is the distribution at time $n$ of the Markov chain corresponding to the algorithm and $\mathrm{d_{TV}}(\mu^{(n)}, \pi)$ is the total variation distance between this distribution at time $n$ and the invariant measure $\pi$ of the chain. Rapid convergence of such a MCMC algorithm means that one can bound $\tau(\epsilon)$ by a polynomial in $\epsilon^{-1}$ and the problem size. The algorithm is said to be torpidly mixing if it is not rapidly mixing. There is an intrinsic relationship between $\tau(\epsilon)$ and the spectral gap of the chain defined by
$$
\mathrm{Gap}((X_n)):=\mathrm{Gap}(P) := 1-\max\{|\lambda_i|, \lambda_i \neq 1\}=: 1-|\lambda_1|,
$$
where we write $\lambda_i$ for the eigenvalues of the transition matrix $P=(P(x,y))_{x,y}$ of the chain $(X_n)$ and have $\lambda_1$ denote the second largest eigenvalue. For this define the Dirichlet form of $P$ by
\begin{equation}
 \mathcal{E}(f,f):=\frac{1}{2}\sum_{x,y\in\Omega} |f(x)-f(y)|^2 P(x,y)\pi(x)
\end{equation}
for any function $f:\Omega\to\R$. If we further define
\begin{equation}
 \Var(f):=\E_\pi(f^2)-(\E_\pi f)^2=\frac12\sum_{x,y\in\Omega} |f(x)-f(y)|^2\pi(x)\pi(y)
\end{equation}
it follows that
$$\Gap(P)=\inf \left\{\frac{\mathcal{E}(f,f)}{\Var(f)}\Big|\E_\pi f^2<\infty, \Var(f)\neq 0\right\}.$$

As a matter of fact, for an irreducible and aperiodic chain the following estimates holds true (see e.g. \cite{sinclairbook}): Let $\underline \pi:=\min_x \pi(x)$ (which is non-zero by the ergodic theorem for Markov chains), then
$$
\tau(\epsilon)\le \frac 1 { \mathrm{Gap}(P)} \log (\frac 1 {\underline \pi \epsilon})
$$
as well as
$$
\tau(\epsilon)\ge \frac {|\lambda_1|}{2 \mathrm{Gap}(P)} \log (\frac 1 {2 \epsilon}).
$$
We can thus control the speed of convergence of the Markov chain (or the MCMC algorithm, respectively), if we control the size of the spectral gap of $P$.

\label{AggregatedTransitionMatrix}
\begin{lemma}[Lemma 3 of \cite{MadrasZhengCW}]
Let $P$ be a Markov chain that is reversible with respect to a probability measure $\pi$ on the finite state space $\S$. Also assume that $P(x,x)\geq \frac12$ for every $x\in\S$. Then $P$ is a positive operator.
\end{lemma}

\begin{lemma}[Poincar\'e inequality, Proposition 1' of \cite{DiaconisStrook_GeometricBoundsMC}]\label{Lemma_Poincare_Inequality}
Let $P$ be an irreducible and reversible Markov chain on a finite state space $\S$.
We associate to $P$ the graph with vertex set $\S$ and edges $\langle x,y \rangle$ if and only if $P(x,y)>0$. For each pair of distinct points $x,y\in \S$, we choose a path $\gamma_{xy}$ from $x$ to $y$, such that a given edge appears at most once in a given path.
Then the second largest eigenvalue $\lambda_1$ of $P$ satisfies
$$\lambda_1=1-\Gap(P)\leq 1-\frac1A$$
where
$$A:=\max_{\langle x,y\rangle} \frac{1}{\pi(x)P(x,y)}\sum_{\gamma_{z_1 z_2}\ni \langle x,y\rangle} |\gamma_{z_1 z_2}| \pi(z_1)\pi(z_2)$$
and $|\gamma_{z_1z_2}|$ denotes the number of edges in the path $\gamma_{z_1z_2}$.
\end{lemma}

\begin{lemma}[Comparison of Dirichlet forms, Theorem 2.1  of \cite{DiaconisSaloffCosteComparison}]\label{DSCComparison}
Let $P, \pi$  and $\tilde P, \tilde \pi$  be reversible Markov chains on a finite state space $\S$, with respective Dirichlet forms $\mathcal{E}$ and $\tilde{\mathcal{E}}$.
For each pair $x\ne y,$ with $\tilde P(x,y)>0$, we fix a path $\gamma_{xy}=(x_0=x, x_1,x_2, \ldots, x_k=y)$, such that $P(x_i,x_{i+1})>0$, of length $| \gamma_{xy}| =k$. Set
$E=\{ (x,y) : P(x,y) >0\},  \tilde E=\{ (x,y) : \tilde P(x,y) >0\}$ and $\tilde E(e)=\{ (x,y)\in \tilde E : e \in \gamma_{xy}\},$ where $e\in E$.
Then
$$\tilde{\mathcal{E}} \le A {\mathcal{E}},$$
where
$$A:=\max_{ (z,w)\in E} \frac{1}{\pi(z)P(z,w)}\sum_{\tilde E(z,w)} |\gamma_{xy}| \tilde\pi(x)\tilde P(x,y).$$
\end{lemma}

\begin{lemma}[Lemma 5 of \cite{MadrasZhengCW}]\label{Lemma_Compare_Derichlet_forms}
Let $(P,\pi)$ and $(\tilde P,\tilde \pi)$ be two Markov chains on the same finite state space $\S$, with respective Dirichlet forms $\mathcal{E}$ and $\mathcal{E}'$. Assume that there exists constants $A,a>0$ such that
$$\mathcal{E}'\leq A\mathcal{E}\quad \text{and}\quad a\pi\leq\tilde\pi.$$
Then
$$\Gap(\tilde P)\leq \frac{A}{a} \Gap(P).$$
\end{lemma}
\begin{remark}
 A sufficient condition for $\mathcal{E}'\leq A\mathcal{E}$ is that
$$\tilde \pi(x)\tilde P(x,y)\leq A\ \pi(x)P(x,y)\quad \text{for all }x,y\in\S\text{ such that }x\neq y.$$
\end{remark}

\begin{lemma}[Lemma 6 of \cite{MadrasZhengCW}]\label{Multiple_usage_does_not_change_Gap}
For any reversible finite Markov chain $P$,
$$\Gap(P)\geq \frac1m \Gap(P^m)\quad \forall m\in\N^*.$$
\end{lemma}

\begin{lemma}[Lemma 7 of \cite{MadrasZhengCW}]\label{MadrasZhengLemma7}
 Let $A$ and $B$ be Markov kernels reversible with respect to a distribution $\pi$. The following holds for $A$ and $B$:
$$\Gap(ABA)\geq \Gap(B).$$
This also holds for $A$ substituted by $A$'s positive square root $A^{\frac12}$, if additionally $A$ is a nonnegative (self-adjoint) operator.
\end{lemma}

\begin{theorem}[Caracciolo-Pelissetto-Sokal \cite{MadrasZhengCW}]\label{CaraccioloPelissettoSokal}
Let $\mu$ be a probability distribution on a finite state space $\mathcal{S}$, and let $\mathcal{P}$ be a transition matrix reversible with respect to $\mu$. Suppose that we partition the set $\mathcal{S}$ as
$$\mathcal{S}= \bigcup_{i=1}^m \mathcal{S}_i,  \mbox{with}\ \mathcal{S}_i\cap\mathcal{S}_j =\emptyset , \mbox{if}\  i\neq j.$$
For each $i=1,\ldots, m$, let $\mathcal{P}_i$ be the restriction of $\mathcal{P}$ to $\mathcal{S}_i$, by rejecting jumps that leave $\mathcal{S}_i$ : for all $x\in \mathcal{S}_i,$  for all $B \subset \mathcal{S}_i$,

$$\mathcal{P}_i(x,B)= \mathcal{P}(x,B) + {\bf 1}_{\{x\in B\}} \mathcal{P}(x, \mathcal{S}\setminus \mathcal{S}_i).$$

 Let $\mathcal{Q}$ be a positive operator, that is also reversible with respect to $\mu$, and $\overline{\mathcal{Q}}$ the aggregated chain associated to the partition $(\mathcal{S}_i)_{i=1,...,m}$; more precisely, for $i,j=1,\ldots,m$,
$$\overline{\mathcal{Q}}(i,j)= \frac 1{\mu(\mathcal{S}_i)} \sum_{x\in\mathcal{S}_i}  \sum_{y\in\mathcal{S}_j} \mu(x) \mathcal{Q}(x,y).$$
 Let $\mathcal{Q}^{\frac12}$ be the positive square root of $\mathcal{Q}$. Then
 \begin{equation}
 \Gap(\mathcal{Q}^{\frac12}\mathcal{P}\mathcal{Q}^{\frac12})\geq \Gap(\overline{\mathcal{Q}})\cdot \min_{1\leq i\leq m}\Gap(\mathcal{P}_i).
\end{equation}

\end{theorem}

\begin{theorem}[Diaconis and Saloff-Coste \cite{DiaconisSaloffCosteComparison}]\label{DiaconisSaloffCosteProductChain}
 For $i=1,...,M$, let $P_i$ be a reversible Markov chain on a finite state space $\Omega_i$. Consider the product Markov chain $P$ on the product space $\Omega_0\times...\times \Omega_M$, defined by
\begin{equation}
 P=\frac1{M+1}\sum_{i=0}^{M}I\otimes...\otimes I\otimes P_i\otimes I\otimes...\otimes I,
\end{equation}
where (in a slight abuse of notation) $I$ denotes the identity on the space it is defined.
Then $\Gap(P)=\frac{1}{M+1}\min_{i\in\{0,...,M\}} \{\Gap(P_i)\}$.
\end{theorem}

\begin{theorem}[Jerrum and Sinclair \cite{JerrumSinclair}]
\label{Conductance_Lemma}
Let $P$ be a Markov chain on a finite set $\Omega$ reversible with respect to $\pi$. For all $\mathcal{S}\subset \Omega$, let
$$ \Phi_{\mathcal{S}} = \frac {\sum_{x\in \mathcal{S}, y\notin \mathcal{S}} \pi(x) P(x,y)}{\pi(\mathcal{S})},$$ and the conductance $\Phi$   given by
$$\Phi  = \min_{\mathcal{S} : \pi(\mathcal{S})\le 1/2} \Phi_\mathcal{S}.$$
Then we have
$$ \frac {\Phi^2}2 \le \Gap(P) \le 2 \Phi.$$
\end{theorem}

\section{Random 3-coloring of the complete graph}\label{rapid_mixing_on_three_colorings}
In this section, we will give a rapidly mixing Markov chain $(X_i)_i$ which has the uniform distribution on the set of of all 3-Colorings with a given number of vertices of a certain color as its stationary distribution. This will be of use, as we intend to compare the Metropolis Algorithm on $\A_{s,r}$ (see \eqref{A_sr}) of the BEG model with this chain in order to show rapid mixing.

Let $\Lambda=\{1,...,N\}$ and define $\Omega=\{-1,0,1\}^\Lambda$ to be the set of all possible 3-colorings of $\Lambda$. Note, that we do not restrict ourselves to 3-colorings in the graph theoretic sense, where adjacent vertices are required to have different colors. Further consider a tuple $(a_1,a_2,a_3)\in\Upsilon$, thus $N a_i$ represents the number of vertices, which have color $i$. Now let
\begin{equation}\mathcal{C}=\Big\{\sigma\in\Omega\Big|\frac{1}{N}\sum_j \delta_{i,\sigma_j}=a_i\Big\}\label{C_uniform_distribution_set}\end{equation}
be the set of appropriate 3-colorings and $\rho$ the uniform distribution on $\mathcal C$. Our aim is to give a Markov chain $(X_i)_{i\in\N}$ which compares well to the chain we  consider  in Section \ref{BEGequienergy} for the BEG model and which also samples efficiently from $\rho$.
\subsection{Rapid mixing of $(X_i)$}
Fix $\Ccal$ as in \eqref{C_uniform_distribution_set}. Consider the Markov chain $(X_i)$ on $\Ccal$ with the following transition kernel. Take $(\mathcal{R}_1(i))_{i\in\N}$ and $(\mathcal{R}_2(i))_{i\in\N}$ independently and uniformly distributed on $\{1,...,N\}$. Define
\begin{align}
X_1&:=X\in\Ccal\notag\\
X_{i+1}&:=\begin{cases}
         X_i&\mathcal{R}_1(i)=\mathcal{R}_2(i)\\
	 \big(\mathcal{R}_1(i),\mathcal{R}_2(i)\big)\big(X_{i}\big)&\mathcal{R}_1(i)\neq \mathcal{R}_2(i)
        \end{cases}
\end{align}
(where $X$ is any admissible starting point and for a vector $x:=(x_1, \ldots, x_N)$ and $i \neq j \in \{1, \ldots, N\}$ we write $(i,j)(x_1, \ldots x_N)$ for the vector $x$ with the components $i$ and $j$ interchanged)
and verify, that $(X_i)$ has reversible distribution $\rho$ on $\Ccal$. We will use a coupling argument in order to show rapid convergence to equilibrium of $(X_i)$. To this end define

\begin{equation} X'_1:=X'\in\Ccal\end{equation}
with $X'$ drawn according to $\rho$ and iteratively
\begin{equation} \Ccal(i):=\big\{j\in\{1,...,N\}\big|X_i(j)\neq X'_i(j)\big\}\end{equation}
with
\begin{align*}
X'_{i+1}&:=\begin{cases}
         X'_i&\mathcal{R}_1(i)=\mathcal{R}_2(i)\\
	 \big(\mathcal{R}_1(i),\mathcal{R}_2(i)\big)\big(X'_{i}\big) &X_i(\mathcal{R}_1(i))=X'_i(\mathcal{R}_1(i)) \wedge X_i(\mathcal{R}_2(i))\neq X'_i(\mathcal{R}_2(i))\\
	 \big(\mathcal{R}_1(i),\mathcal{R}_2(i)\big)\big(X'_{i}\big) &X_i(\mathcal{R}_1(i))\neq X'_i(\mathcal{R}_1(i)) \wedge X_i(\mathcal{R}_2(i))= X'_i(\mathcal{R}_2(i))\\
	\big(\mathcal{R}_1(i),\mathcal{R}_2(i)\big)\big(X'_{i}\big) &X_i(\mathcal{R}_1(i))= X'_i(\mathcal{R}_1(i)) \wedge X_i(\mathcal{R}_2(i))= X'_i(\mathcal{R}_2(i))\\
	 \big(\mathcal{R}_1(i),\mathcal{R}_3(i)\big)\big(X'_{i}\big) &\text{otherwise}
        \end{cases}
\end{align*}
and $\mathcal{R}_3$ being uniformly drawn out of $\Ccal(i)$ and independent of $(\mathcal{R}_1(i))$ and $(\mathcal{R}_2(i))$. Again verify that $(X'_i)$ is a Markov chain which is reversible with respect to $\rho$ on $\Ccal$. Thus $(X'_i)$ is in equilibrium in every step.
\begin{lemma}\label{Coupling_Lemma}
The expected coupling time $T_\Ccal$ of the Markov chains $(X_i)$ and $(X'_i)$ is bounded from above by
$$\E T_\Ccal\leq N^4.$$
\end{lemma}
\begin{proof}
Define $\Psi(i):=|\Ccal(i)|$. Once $\Psi(i)=0$ the two chains have coupled. Due to the construction $\Psi$ is monotonically decreasing. Indeed, if $X_i(k)=X_i'(k)$ holds for one $i$ and a $k \in \{1, \ldots, N\}$, we will
have $X_j=X_j'$ for the position $k$ is  permuted to.
We further know
$$\P\big(\Psi(i+1)\leq j-1\big|\Psi(i)=j>0\big)\geq \frac{1}{N^3},$$
as all that needs to happen is, find two components $k_1$ and $k_2$
{ such that the chains differ at both positions and the number of differences can be reduced by at least one through exchanging spins in one of the chains. Such $k_1$ and $k_2$ always exist
 and we can }choose these with $\mathcal{R}_1$ and $\mathcal{R}_2$ which happens with probability $\frac{1}{N^2}$. In this case $\mathcal{R}_3$ would be drawn out of all components in which $X_i$ and $X_i'$ differ. There are at most $N$ of those. Using \cite[Chapter 4-3, Lemma 1]{AldousFill:book} we get an upper bound of
$$\E T_\Ccal\leq \sum_{i=1}^{N} N^3=N^4$$
for the coupling time.
\end{proof}

\section{Existence of $\Klow$}\label{existenceKlow}
\newcommand{\Kco}{K_c^{(1)}}
\newcommand{\Kct}{K_c^{(2)}}
As \cite{EllisOttoTouchetteBEG} did not completely prove the existence of $\Klow:=\lim_{\beta \rightarrow +\infty} \Kco(\beta)$, we will do so in this section.
\begin{lemma}\label{Lemma_Kco_continuous}
The function
\begin{align*}
\Kco:& (\beta_c,\infty) &\longrightarrow &&\R\notag\\
&\beta &\mapsto&& \Kco(\beta)\notag
\end{align*}
is continuous.
\end{lemma}
\begin{proof}
As shown in \cite{EllisOttoTouchetteBEG} $\Kco(\beta_c)=\Kct(\beta_c)$. It is also shown that $\Kct$ is continuous and monotonically decreasing on its domain.

Assume $\Kco$ to not be continuous. Then there exists a $\beta_d\geq\beta_c$ such that either $\Kco$ is discontinuous at $\beta_d=\beta_c$ or such that $\Kco$ is discontinuous at $\beta_d$ and continuous for all $\beta\in[\beta_c,\beta_d)$. Then there exists a monotonic sequence $(\beta_i)_i$ with $\beta_i\neq \beta_d$, $\lim \beta_i=\beta_d$ and $\lim\Kco(\beta_i)\neq \Kco(\beta_d)$.
\begin{enumerate}
 \item Suppose first that $\lim\Kco(\beta_i) < \Kco(\beta_d)$. Fix $K_d\in(\lim\Kco(\beta_i), \Kco(\beta_d))$. The analysis given by Ellis et al. in \cite{EllisOttoTouchetteBEG} guarantees the BEG state space for $(K_d,\beta_d)$ to have exactly one macrostate while for all but finitely many $i$ the BEG state space for $(K_d,\beta_i)$ has exactly two modes. Have $f_\beta$ as defined in (\ref{EnergyLandscape}). It is smooth and clearly, for $K=K_d$, we have the functional limit $$\lim_{\beta\uparrow\beta_d} f_{\beta}=f_{\beta_d}.$$ Thus in this case $f_{\beta_d}$ has either exactly one global maximum or exactly three global maxima.
 \item The second case for $\lim\Kco(\beta_i) > \Kco(\beta_d)$ works the same.
\end{enumerate}
\end{proof}

\begin{lemma}
The function
\begin{align*}
\Kco:& (\beta_c,\infty) &\longrightarrow &&\R\notag\\
&\beta &\mapsto&& \Kco(\beta)\notag
\end{align*}
is monotonic.
\end{lemma}
\begin{proof}
Assume $\Kco$ not to be monotonic. Then there exist $\beta_1<\beta_2<\beta_3<\beta_4$ such that $\Kco(\beta_4)>\Kco(\beta_1)=\Kco(\beta_3)>\Kco(\beta_2)$ as $\Kco$ is continuous as shown in Lemma \ref{Lemma_Kco_continuous}. This guarantees that the BEG model has at least two phase transitions for $\Kco(\beta_1)$. With the analysis done by Ellis et al. in \cite{EllisOttoTouchetteBEG} it is however clear where exactly the macrostates lie. Thus the first phase transition of the model must switch from one to two modes and the second the model exhibits for growing $\beta$ must change back to exactly one mode. This is in clear violation of Lemma \ref{Lemma_monotonc-behaviour-of-A_g} of this paper.
\end{proof}

\begin{corollary} \label{inverse}
The proof given in Section 5 by Ellis et al. in \cite{EllisOttoTouchetteBEG} is correct if $\Kco$ is inverted only on $\operatorname{Im}(\Kco(\beta_c,\infty))$.
\end{corollary}

\begin{corollary} \label{Klow}
The limit of $\Kco(\beta)$ as $\beta \rightarrow +\infty$ exists.
\end{corollary}


\section{Analysis of $f_\beta$}
This appendix contains a detailed analysis of the function $f_\beta$ given in \eqref{EnergyLandscape}.
The first result  we prove in this appendix is the Theorem \ref{Theorem_Holger}.
We first change coordinates.
Let $r=\frac{x}{x+z}$ and $t=x+z$. Then the mapping is
\[
T:\Upsilon_\infty \to(0,1)^2\mbox{ with }
(a_{-1},a_0,a_1)\mapsto(r,t)
\]
bijective. Hence, instead of investigating the maxima of $f_\beta$, we can analyze the
minima of $F(r,t):= F_\beta(r,t) :=-f_\beta\circ T^{-1}(r,t)$.
Here $F:(0,1)^2\to \R$ is given by
\[
F(r,t)=\beta t(1-Kt(1-2r)^2)+tH(r)+H(t),
\]
with $H(r)=r\log r+(1-r)\log(1-r)$.

\medskip

\textbf{Minimums at the boundary:} For fixed $r\in[0,1]$ the function $F$ is the sum of a polynomial in $t$ and  the entropy function $H(t)$. Now $H(t)$ is steep at $t=0$ and $t=1$, hence there are no local minima in these points.

If, on the other hand, $t\in(0,1)$ is fixed, the same argument yields that there are no local minima in $r=0$ and $r=1$,
either.

\medskip

\textbf{Global and local Minimums:} We take derivatives of $F$ for $r,t\in(0,1)$.
\[
\begin{array}{ll}
\partial_rF(r,t)=&4\beta Kt^2(1-2r)+t\log\frac{r}{1-r}\\
\partial_tF(r,t)=&\beta-2\beta Kt(1-2r)^2+H(r)+\log\frac{t}{1-t}\\
\partial_r^2F(r,t)=&-8\beta Kt^2+\frac{t}{r(1-r)}\\
\partial^2_{rt}F(r,t)=&8\beta Kt(1-2r)+\log\frac{r}{1-r}\\
\partial^2_tF(r,t)=&-2\beta K(1-2r)^2+\frac{1}{t(1-t)}
\end{array}
\]
Hence the equations for potential minima are
\begin{equation}\label{eins}
4\beta Kt(2r-1)=\log\frac{r}{1-r}
\end{equation}
\begin{equation}\label{zwei}
\frac{1}{t}-1=e^\beta\cdot\sqrt{r(1-r)}\ ,
\end{equation}
where we have used \eqref{eins} to solve $\partial_tF=0$ and obtain \eqref{zwei}.
Taking the Taylor expansion of $F$ in a critical point $(r_0,t_0)$ up to second order we see that  $$F(r,t)=F(r_0,t_0)+\frac{1}{2}A$$
where
\[
A=\partial_r^2F(r_0,t_0)(r-r_0)^2+2\partial_{rt}^2F(r_0,t_0)(r-r_0)(t-t_0) +\partial^2_tF(r_0,t_0)(t-t_0)^2.
\]
Putting $w:=\sqrt{r_0(1-r_0)}$ we see that $t_0=(1+e^\beta w)^{-1}$ and therefore
\begin{equation}\label{rzwei}
\partial_r^2F(r_0,t_0)=\frac{t_0^2}{w^2}(1+e^\beta w-8\beta Kw^2).
\end{equation}
Due to \eqref{eins} we have in critical points $(r_0, t_0)$
\[
\partial^2_{rt}F(r_0,t_0)=4\beta Kt_0(1-2r_0)
\]
and the determinant of the Hessian $M$ in $(r_0,t_0)$ is given by
\[
\det M=\Big(\frac{t_0}{w^2}-8\beta Kt_0^2\Big)\Big(\frac{1}{t_0(1-t_0)}-2\beta K(1-4w^2)\Big)
-(4\beta Kt_0)^2(1-4w^2).
\]
This can be simplified to
\begin{eqnarray*}
\det M&=&\Big(\frac{1}{w^2}-8\beta Kt_0\Big)\frac{1}{1-t_0}-
2\beta K\frac{t_0}{w^2}(1-4w^2)\\
&=&\frac{1-2\beta K t_0+2\beta Kt_0^2(1-4w^2)}{w^2(1-t_0)}
\end{eqnarray*}
and by replacing $t_0$ we obtain:
\begin{eqnarray}\label{signum}
\det M&=&\frac{(1+e^\beta w)^2-2\beta K(1+e^\beta w)+2\beta K(1-4w^2)}{w^2(1-t_0)(1+e^\beta w)^2}\nonumber\\
&=&\frac{1+2e^\beta w(1-\beta K)+w^2(e^{2\beta}-8\beta K)}{w^2(1-
t_0)(1+e^\beta w)^2}.
\end{eqnarray}
Note that the sign of $\det M$ is determined by the sign of the nominator, which is important, since $M$ is positive
definite in $(r_0, t_0)$, if $\partial_r^2F>0$ and $\det M>0$ in that point.

Investigating which points are critical, we see the following

\begin{enumerate}
\item Obviously, $r_0=\frac{1}{2}$ , $t_0=\frac{2}{2+e^\beta}$ is critical. Here $\partial_r^2F(r_0,t_0)=2t_0^2(2+e^\beta-4\beta K)$ and hence
\[
A=2t_0^2(2+e^\beta-4\beta K)(r-r_0)^2+\frac{1}{t_0(1-t_0)}(t-t_0)^2.
\]
Thus there is a local minimum of $F$ in $(r_0,t_0)$, if and only if $4\beta K\le2+e^\beta$.
If $4\beta K>2+e^\beta$, $(r_0, t_0)$ as defined above is not an extremal point.
\item For $r\ne\frac{1}{2}$, we only consider $r\in I:=(\frac{1}{2},1)$, since $F$ is symmetric in $r$ around $\frac{1}{2}$.

Combining \eqref{eins} and \eqref{zwei} we see that a necessary condition for $(r,t)$ to be a local minimum is
\begin{equation}\label{drei}
h(r):=\log\frac{r}{1-r}=\frac{4\beta K(2r-1)}{1+e^\beta\sqrt{r(1-r)}}
:=\phi(r),
\end{equation}
which we will investigate for solutions in $I$. Let $w(r):=\sqrt{r(1-r)}$. We compute

\begin{eqnarray*}
h'(r)=\frac{1}{r}+\frac{1}{1-r}=\frac{1}{w^2(r)}\\
h''(r)=-\frac{1}{r^2}+\frac{1}{(1-r)^2}=\frac{2r-1}{w^4(r)}
\end{eqnarray*}
and
\[
\phi'(r)=4\beta K\frac{2+2e^\beta w(r)-(2r-1)e^\beta\frac{(1-2r)}{2w(r)}}{
(1+e^\beta w(r))^2}=2\beta K\frac{4w(r)+e^\beta}{w(r)(1+e^\beta w(r))^2}
\]
and eventually
\begin{eqnarray*}
\phi''(r)\!\!\!\!\!&=&\!\!\!2\beta K\frac{4w'(r)w(r)(1+e^\beta w(r))^2-(4w(r)+e^\beta)
[w(r)(1+e^\beta w(r))^2]'}{w^2(r)(1+e^\beta w(r))^4}\\
&=&\!\!\!2\beta Kw'(r)\frac{4w(r)(1+e^\beta w(r))-(4w(r)+e^\beta)[
1+e^\beta w(r)+2w(r)e^\beta]}{w^2(r)(1+e^\beta w(r))^3}\\
&=&\!\!\!\beta Ke^\beta\frac{(2r-1)(8w^2(r)+3e^\beta w(r)+1)}{w^3(r)(1+e^\beta w(r))^3}.
\end{eqnarray*}

Now $h'(r){\substack{<\\=\\>} }\phi'(r)$ implies
\begin{equation}\label{erst}
(e^{2\beta}-8\beta K)w^2(r)+2e^\beta(1-\beta K)w(r)+1{\substack{<\\=\\>} }0.
\end{equation}
Hence there are at most two solutions $r_1,r_2\in I$ with $\phi'=h'$, because $w$ is injective on $I$.
Therefore, according to Rolle's theorem also the equation $\phi=h$ has at most two further solutions in $I$ (next to $r=1/2$). Moreover, we see that the left hand side of \eqref{erst} equals the nominator of $\det M$ in \eqref{signum}. In a critical point we thus have $h'<\phi'$ (or $h'>\phi'$, respectively) if and only if in this point it holds $\det M<0$ (or $\det M>0 $, respectively).

Again we distinguish different cases:

If $4\beta K>2+e^\beta$, then $\phi '(1/2)>h'(1/2)$ and thus $\phi>h$ on $(1/2,1/2+\delta)$ for an appropriate $\delta>0$. Now, close to $r=1$ we  always have $\phi<h$, which means, there is at least one solution $\phi=h$ in $I$. However, there cannot be two such solutions: If there were $\frac{1}{2}<r_1<r_2<1$ with $\phi=h$, then
$\phi-h$ cannot change sign in both solutions, otherwise we would have $\phi>h$ also in a right neighborhood of $r_2$ and we would need a third solution $r_3$ to the right of $r_2$, in contradiction to the above conclusion.
If, on the other hand, $\phi-h$ cannot change sign in both solutions, then at least one of $r_1$ and $r_2$ also solves $\phi'=h'$. But this again leads to a contradiction. Again using Rolle's theorem we see that $\phi=h$ for
$\frac{1}{2}<r_1<r_2$ implies that there exist $\xi_1,\xi_2$ with $\phi'=h'$ and
\[
\frac{1}{2}<\xi_1<r_1<\xi_2<r_2
\]
and there cannot be more than two solutions of $\phi'=h'$.

Hence there is exactly one solution $r_1\in I$ and from \eqref{zwei} one obtains the corresponding $t_1$, such that  $(r_1,t_1), (1-r_1,t_1)$ and $(r_0,t_0)$ are the only critical points of $F$.
However, we already know that here we have $4\beta K>2+e^\beta$ and hence $(r_0,t_0)$ is not a minimum of $F$. Moreover, minima at the boundary do not exist. But $F$ is continuous on $[0,1]^2$, therefore has a minimum, thus  the points $(r_1,t_1)$ and $(1-r_1,t_1)$ are global minima.

If, on the other hand $4\beta K=2+e^\beta$ and $e^\beta>4$, then $ \phi '(1/2)=h'(1/2)$ and of course $\phi''(1/2)=h''(1/2)=0$, however we still have $ \phi '''(1/2)>h'''(1/2)$, hence again $ \phi >h$ on $(1/2,1/2+\delta)$ for an appropriate $\delta>0$.
$\phi '''(1/2)>h'''(1/2)$ can be seen as follows: Write
\[
v(u):=\frac{8u^2+3e^\beta u+1}{(u+e^\beta u^2)^3}.
\]
Then $\phi''(r)=\beta K e^\beta (2r-1)v\circ w(r)$ and hence
\begin{equation}\label{abl}
\phi'''(r)=\beta K e^\beta(2v\circ w(r)-(2r-1)^2\frac{1}{2w(r)}v'\circ w(r)).\end{equation}
Thus
\[
\phi'''(1/2)=\frac{1}{2}(2+e^\beta)e^\beta v(1/2)
=48\frac{e^\beta}{2+e^\beta}.
\]
Due to $h'''(1/2)=32$ we have $\phi '''(1/2)>h'''(1/2)$ if and only if $e^\beta >4$.

Analogously to our arguments above we see that there is only one solution $r_1\in I$ of $\phi=h$, and again the corresponding $t_1$ can be computed from \eqref{zwei}. Indeed there is a local minimum of $F$ in $(r_1,t_1)$ and $(1-r_1,t_1)$. This can be seen by showing that the Hessian is positive definite. However, as this is not part of our assertion, we will refrain from doing so.

If, finally $4\beta K=2+e^\beta$ and $e^\beta\le4$, then $\phi'(1/2)=h'(1/2)$ and $\phi'''(1/2)\le h'''(1/2)$ and $\phi^{(5)}(1/2)<h^{(5)}(1/2)$, such that again $\phi<h$ on $(1/2,1/2+\delta)$ for an appropriate $\delta>0$.

For $\phi^{(5)}(1/2)<h^{(5)}(1/2)$ one argues: Because of \eqref{abl} we have
\begin{eqnarray*}
\phi^{(5)}(1/2)&=&\beta K e^\beta (2(v\circ w)''(1/2)-8v'(1/2))\\
&=&2\beta K e^\beta\Big(((v'\circ w)\cdot w')'(1/2)-4v'(1/2)\Big)\\
&=&2\beta K e^\beta\Big(-\frac{v'\circ w}{w}(1/2)-4v'(1/2)\Big)\\
&=&-12\beta K e^\beta v'(1/2)
\end{eqnarray*}
and
\begin{eqnarray*}
v'(1/2)&=&\frac{(8+3e^\beta)\frac{1}{4}(2+e^\beta)
-3(1+e^\beta)(3+\frac{3}{2}e^\beta)}{(\frac{1}{2}+\frac{1}{4}e^\beta)^4}\\
&=&-320\frac{2+3e^\beta}{(2+e^\beta)^3},
\end{eqnarray*}
thus
\[
\phi^{(5)}(1/2)=960e^\beta\frac{2+3e^\beta}{(2+e^\beta)^2}.
\]
Because of $h^{(5)}(1/2)=4!\cdot 2^6$ one has $\phi^{(5)}(1/2)<h^{(5)}(1/2)$ if and only if
$5e^\beta(2+3e^\beta)<8(2+e^\beta)^2$, thus $7e^{2\beta}-22e^\beta-32<0$ and this is true for all $0<e^\beta\le 4$.

The same is of course also true, when $4\beta K<2+e^\beta$, since then we already have $\phi'(1/2)<h'(1/2)$.

Summarizing we see that in all possible cases we have at most three local minima of $F$ and none at the boundary.
Of course, we could discuss how many minima there are exactly in certain cases. However, we will refrain from doing so, since this is not needed.

\end{enumerate}

\vspace{2mm}
 The second result we prove in this appendix is needed for the slow convergence case.
\begin{lemma}\label{N_is_unimodal}
There exists an $\epsilon_0>0$ such that for any $0<\epsilon\leq\epsilon_0$ on the set $$\mathcal{N}=\big\{\sigma\big||S_N(\sigma)|\leq N\cdot\epsilon\big\}$$ as defined in \eqref{N_central},
the free energy $f_\beta$ is unimodal for all $\beta$.
\end{lemma}
\begin{proof}
The claim is true, if we find an $\epsilon_0>0$ such that $f_\beta((a_{-1},a_0,a_1))$ is unimodal on $|a_{-1}-a_1|<\epsilon_0$. Consider
\begin{align}
-f_\beta(a_1,a_0,a_1)&=2\beta a_1+2a_1 \log(a_1)+(1-2a_1)\log(1-2a_1)\\
-f_\beta(a_1,a_0,a_1)'&=2\beta-\log\left(\frac{1}{a_1}-2\right)
\end{align}
which tells us, that there is exactly one mode on the $a_1=a_{-1}, a_0=1-2a_{1}$ line. As $f_\beta$ is smooth this generalizes for all lines $a_1=a_{-1}+2\epsilon_0$ for sufficiently small $\epsilon_0$. This yields the desired result by using Theorem \ref{Theorem_Holger} as all that could happen, are maxima on the boundary.
\end{proof}


\vspace{0.3cm}
{\bf Acknowledgement}: We are very grateful to the anonymous referees for
 valuable comments, in particular to have informed us of the recent paper by Kovchegov, Otto and Titus \cite{Otto}.

\bibliographystyle{abbrv}
\bibliography{LiteraturDatenbank}
\end{document}